\documentclass[reqno,11pt]{amsart}
\usepackage{amsmath,amssymb,latexsym,soul,cite,mathrsfs}
\usepackage{color,enumitem,graphicx}
\usepackage[colorlinks=true,urlcolor=blue,
citecolor=red,linkcolor=blue,linktocpage,pdfpagelabels,
bookmarksnumbered,bookmarksopen]{hyperref}
\usepackage[english]{babel}
\usepackage[left=2.6cm,right=2.6cm,top=2.9cm,bottom=2.9cm]{geometry}
\usepackage[hyperpageref]{backref}

\newtheorem{theorem}{Theorem}
\newtheorem{lemma}[theorem]{Lemma}
\newtheorem{corollary}[theorem]{Corollary}
\newtheorem{proposition}[theorem]{Proposition}
\newtheorem{remark}[theorem]{Remark}
\newtheorem{conjecture}[theorem]{Conjecture}
\newtheorem{definition}[theorem]{Definition}

\newtheorem{theoremletter}{Theorem}

\newenvironment{acknowledgement}{\noindent\textbf{Acknowledgments}}{}

\newtheoremstyle{tttheorem}
{}                
{}                
{\slshape}        
{}                
{\bfseries}       
{'}               
{ }               
{}                
\theoremstyle{tttheorem}

\newcommand{\ud}{\mathrm{d}}

\newcommand*{\avint}{\mathop{\ooalign{$\int$\cr$-$}}}

\title[Asymptotics for solutions to subcritical sixth order equations]{Asymptotics for positive singular solutions to subcritical sixth order equations} 
\thanks{This research is partially supported by S\~ao Paulo Research Foundation (FAPESP) \#2020/07566-3 and \#2021/15139-0 and Natural Sciences and Engineering Research Council of Canada (NSERC)}

\author[J.H. Andrade]{Jo\~{a}o Henrique Andrade}
\author[J. Wei]{Juncheng Wei}

\address[J.H. Andrade]{
	Department of Mathematics,
	University of British Columbia
	\newline\indent 
	V6T 1Z2, Vancouver-BC, Canada
	\newline\indent
	and
	\newline\indent 
	Institute of Mathematics and Statistics,
	University of S\~ao Paulo
	\newline\indent 
	05508-090, S\~ao Paulo-SP, Brazil
}
\email{\href{mailto:andradejh@math.ubc.ca}{andradejh@math.ubc.ca}}
\email{\href{mailto:andradejh@ime.usp.br}{andradejh@ime.usp.br}}

\address[J. Wei]{
	Department of Mathematics,
	University of British Columbia
	\newline\indent 
	V6T 1Z2, Vancouver-BC, Canada}
\email{\href{mailto:jcwei@math.ubc.ca}{jcwei@math.ubc.ca}}


\begin{document}
	
	\begin{abstract}
		We classify the local asymptotic behavior of positive singular solutions to a class of subcritical sixth order equations on the punctured ball.
		Initially, using a version of the integral moving spheres technique, we prove that solutions are asymptotically radially symmetric solutions with respect to the origin.
		We divide our approach into some cases concerning the growth of nonlinearity.
		In general, we use an Emden--Fowler change of variables to translate our problem to a cylinder.
		In the lower critical regime, this is not enough, thus, we need to introduce a new notion of change of variables.
		The difficulty is that the cylindrical PDE in this coordinate system is nonautonomous.
		Nonetheless, we define an associated nonautonomous Pohozaev functional, which can be proved to be asymptotically monotone.
		In addition, we show {\it a priori} estimates for these two functionals, from which we extract compactness properties.
		With this ingredients, we can perform an asymptotic analysis technique to prove our main result.
	\end{abstract}
	
	\maketitle
	
	\bigskip
	\begin{center}
		\footnotesize
		\tableofcontents
	\end{center}
	
	\section{Introduction}
	We study (classical) positive singular  solutions $u\in C^{6}(\mathbb{R}^n\setminus\{0\})$ with $n\geqslant7$ (which will always be assumed so forth) to the following family of subcritical sixth order PDEs 
	\begin{flalign}\tag{$\mathcal P_{6,p,R}$}\label{ourPDE}
		(-\Delta)^3u=f_p(u) \quad {\rm in} \quad B^*_R.
	\end{flalign}
	Here $B_R\setminus\{0\}\subset\mathbb R^n$ is the punctured ball of radius $R>0$, $\Delta^3=\Delta\circ\Delta\circ\Delta$ is the tri-Laplacian, and the nonlinearity $f\in C^{1}(B_R)$ is given by
	\begin{equation*}
		f_p(u):=|u|^{p-1}u \quad {\rm with} \quad p\in(1,2_{\#}]\cup(2_{\#},2^{\#}-1),
	\end{equation*}
	where $2_{\#}:=\frac{n}{n-6}$ and $2^{\#}:=\frac{2n}{n-6}$ are, respectively, the lower and upper critical exponents in the sense of the compact Sobolev embedding of $H^{3}(\mathbb R^n)$.
	
	We say that a positive solution $u\in C^{6}(\mathbb{R}^n\setminus\{0\})$ has a removable singularity at the origin if $\lim_{x\rightarrow0}u(x)<+\infty$, that is, it can be continuously extended to the origin; otherwise, we say that it is a non-removable singularity.
	These are called non-singular and singular solutions, respectively. 
	
	Let us mention that the S. Luo et al. \cite{MR4257807,10.1093/imrn/rnab212} studied homogeneous solutions to \eqref{ourPDE} in the blow-limit case $R=+\infty$, that is,  
	\begin{flalign}\tag{$\mathcal P_{6,p,\infty}$}\label{ourlimitPDE}
		(-\Delta)^3u=f_p(u) \quad {\rm in} \quad \mathbb R^n\setminus\{0\}.
	\end{flalign}
	On this subject, they proved the following classification result
	\begin{theoremletter}\label{thm:luo-wei}
		Let $u\in C^{6}(\mathbb{R}^n\setminus\{0\})$ be a positive singular  solution to \eqref{ourlimitPDE} with $p\in(1,2^{\#}-1)$. 
		Assume that $u$ is homogeneous of degree $-\gamma_p$, where $\gamma_{p}:=\frac{6}{p-1}$.
		\begin{itemize}
			\item[{\rm (a)}] If $p\in(1,2_{\#}]$, then $u\equiv0$.
			\item[{\rm (b)}] If $p=(2_{\#},2^{\#}-1)$, then 
			\begin{equation}\label{subcriticalblowupsolutions}
				u(x)={K}_0(n,p)^{\frac{1}{p-1}}|x|^{-\gamma_p},
			\end{equation}
			where
			\begin{equation*}
				{K}_0(n,p)=\gamma_p\left(\gamma_p+2\right)\left(\gamma_p+4\right)\left(n-2-\gamma_p\right)\left(n-4-\gamma_p\right)\left(n-6-\gamma_p\right).
			\end{equation*}
		\end{itemize}
	\end{theoremletter}
	
	In the light of the last theorem, it makes sense to divide our analysis into three cases, namely, the Serrin--Lions case $p\in(1,2_{\#})$, Aviles case $p=2_{\#}$, and the Gidas--Spruck case $p\in(2_{\#},2^{\#}-1)$.
	Our main result in this manuscript classifies the local behavior of positive solutions to \eqref{ourPDE} in these situations.
	\begin{theorem}\label{maintheorem}
		Let $u\in C^6(\mathbb R^n\{0\})$ be a positive singular solution to \eqref{ourPDE}.
		Assume that $-\Delta u\geqslant 0$ and $\Delta^2 u\geqslant 0$.
		Then, it follows
		\begin{equation*}
			u(x)=(1+\mathcal{O}(|x|))\overline{u}(x) \quad {\rm as} \quad x\rightarrow0,
		\end{equation*}
		where $\overline{u}(r)=\avint_{\partial B_{R}}|u(r \theta)|\ud\theta$ is the spherical average of $u$. 
		Moreover, 
		\begin{itemize}
			\item[{\rm (a)}] if $p\in(1,2_{\#})$, then
			\begin{equation*}
				u(x)\simeq|x|^{6-n} \quad {\rm as} \quad x\rightarrow0;
			\end{equation*}
			\item[{\rm (b)}] if $p=2_{\#}$, then 
			\begin{equation*}
				u(x)=(1+ \mathrm{o}(1))\widehat{K}_{0}(n)^{\frac{n-6}{6}}|x|^{6-n}(\ln|x|)^{\frac{6-n}{6}} \quad {\rm as} \quad x\rightarrow0,
			\end{equation*}
			where
			\begin{equation}\label{serrinasymptoticsconstant}
				\widehat{K}_{0}(n)=\frac{4}{3}(n-2)(n-4)(n-6)^2;
			\end{equation}
			\item[{\rm (c)}] if $p\in(2_{\#},2^{\#}-1)$, then 
			\begin{equation*}
				u(x)=(1+\mathrm{o}(1))K_0(n,p)^{\frac{1}{p-1}}|x|^{-\frac{6}{p-1}} \quad {\rm as} \quad x\rightarrow0.
			\end{equation*}
		\end{itemize}
	\end{theorem}
	
	When $p\in(1,2_{\#})$, we show that solutions behave like the fundamental solution, and so, the origin is a removable singularity. 
	When $p\in(2_{\#},2^{\#}-1)$, we prove that solutions to \eqref{ourPDE} behave near the isolated singularity like the homogeneous solutions to the blow-up limit equation \eqref{ourlimitPDE}, which are classified by \eqref{subcriticalblowupsolutions}.
	In the lower critical case $p=2_{\#}$, which is the so-called sixth order Serrin exponent, 
	we observe that since $K_0(n,2_{\#})=0$, it follows that \eqref{ourlimitPDE} does not have non-trivial homogeneous solutions.
	This explains why this situation has different blow-up rate near the singularity, which is given by a homogeneous term times a log-correction factor.
	
	Now let us compare our results to the ones in the fourth and second order cases. 
	
	First, we consider positive solutions $u\in C^{4}(\mathbb{R}^n\setminus\{0\})$ with $n\geqslant 5$ to the family of fourth order equations
	\begin{flalign}\tag{$\mathcal P_{4,R,p}$}\label{ourPDE4th}
		(-\Delta)^2u=f_p(u) \quad {\rm in} \quad B_R^*,
	\end{flalign}
	where $R<+\infty$, $\Delta^2=\Delta\circ\Delta$ is the bi-Laplacian, and $p\in(1,2_{**}]\cup(2_{**},2^{**}-1)$, where $2_{**}=\frac{n}{n-4}$ and $2^{**}=\frac{2n}{n-4}$.
	Notice that \eqref{ourPDE4th} is subcritical in the sense of the compact Sobolev embedding $H^{2}(\mathbb R^n)$.
	On this subject, we should mention that R. Soranzo \cite{MR1436822} for $p\in(1,2_{**})$, H. Yang \cite{MR4123335} and Z. Guo et al. \cite{MR3632218} for $p\in(2_{**},2^{**}-1)$,
	and the first-named author and J. M. do \'O \cite{arXiv:2009.02402} for $p=2_{**}$
	study qualitative properties for positive solutions to \eqref{ourPDE4th}. We have the result below
	\begin{theoremletter}
		Let $u\in C^4(\mathbb R^n\{0\})$ be a positive singular solution to \eqref{ourPDE4th}.
		Assume that $-\Delta u\geqslant 0$.
		Then, it follows
		\begin{equation*}
			u(x)=(1+\mathcal{O}(|x|))\overline{u}(x) \quad {\rm as} \quad x\rightarrow0.
		\end{equation*}
		Moreover, 
		\begin{itemize}
			\item[{\rm (a)}] if $p\in(1,2_{**})$, then
			\begin{equation*}
				u(x)\simeq|x|^{4-n} \quad {\rm as} \quad x\rightarrow0;
			\end{equation*}
			\item[{\rm (b)}] if $p=2_{**}$, then 
			\begin{equation*}
				u(x)=(1+ \mathrm{o}(1))\widehat{K}_{4,0}(n)^{\frac{n-4}{4}}|x|^{4-n}(\ln|x|)^{\frac{4-n}{4}} \quad {\rm as} \quad x\rightarrow0,
			\end{equation*}
			where
			\begin{equation*}
				\widehat{K}_{4,0}(n)=\frac{(n-2)(n-4)^2}{2};
			\end{equation*}
			\item[{\rm (c)}] if $p\in(2_{**},2^{**}-1)$, then 
			\begin{equation*}
				u(x)=(1+\mathrm{o}(1))K_{4,0}(n,p)^{\frac{1}{p-1}}|x|^{-\frac{4}{p-1}} \quad {\rm as} \quad x\rightarrow0.
			\end{equation*}
			where
			\begin{align*}
				K_{4,0}(n,p)=\frac{4}{p-1}\left(\frac{4}{p-1}+2\right)\left(n-2-\frac{4}{p-1}\right)\left(n-4-\frac{4}{p-1}\right).
			\end{align*}
		\end{itemize}
	\end{theoremletter}
	
	Second, we consider positive singular solutions $u\in C^{2}(\mathbb{R}^n\setminus\{0\})$ with $n\geqslant 3$ the second order equation below
	\begin{flalign}\tag{$\mathcal P_{2,R,p}$}\label{ourPDE2th}
		-\Delta u=f_p(u) \quad {\rm in} \quad B_R^*,
	\end{flalign}
	where $R<+\infty$, $\Delta$ is the Laplacian, and $p\in(1,2_{*}]\cup(2_{*},2^{*}-1)$, where $2_{*}=\frac{n}{n-2}$ and $2^{*}=\frac{2n}{n-2}$
	Notice that \eqref{ourPDE2th} is subcritical in the sense of the compact Sobolev embedding of $H^{1}(\mathbb R^n)$.
	All the aforementioned classification results were inspired by the classical theorems of J. Serrin \cite{MR0170096} and P.-L. Lions \cite{MR605060}, B. Gidas and J. Spruck \cite{MR615628}, and P. Aviles \cite{MR875297} on the study of positive singular solutions to the second order semi-linear PDE \eqref{ourPDE2th}
	\begin{theoremletter}
		Let $u\in C^2(\mathbb R^n\{0\})$ be a positive singular solution to \eqref{ourPDE2th}.
		Then, it follows
		\begin{equation*}
			u(x)=(1+\mathcal{O}(|x|))\overline{u}(x) \quad {\rm as} \quad x\rightarrow0,
		\end{equation*}
		Moreover, 
		\begin{itemize}
			\item[{\rm (a)}] if $p\in(1,2_{*})$, then
			\begin{equation*}
				u(x)\simeq|x|^{2-n} \quad {\rm as} \quad x\rightarrow0;
			\end{equation*}
			\item[{\rm (b)}] if $p=2_{*}$, then 
			\begin{equation*}
				u(x)=(1+ \mathrm{o}(1))\widehat{K}_{2,0}(n)^{\frac{n-2}{2}}|x|^{2-n}(\ln|x|)^{\frac{2-n}{2}} \quad {\rm as} \quad x\rightarrow0,
			\end{equation*}
			where
			\begin{equation*}
				\widehat{K}_{2,0}(n)=\frac{(n-2)^2}{2};
			\end{equation*}
			\item[{\rm (c)}] if $p\in(2_{*},2^{*}-1)$, then 
			\begin{equation*}
				u(x)=(1+\mathrm{o}(1))K_{2,0}(n,p)^{\frac{1}{p-1}}|x|^{-\frac{2}{p-1}} \quad {\rm as} \quad x\rightarrow0,
			\end{equation*}
			where
			\begin{align*}
				K_{2,0}(n,p)=\frac{2}{p-1}\left(n-2-\frac{2}{p-1}\right).
			\end{align*}
		\end{itemize}
	\end{theoremletter}
	This type of asymptotic analysis extends to a rich class of strongly coupled second order systems as studied in \cite{MR4085120,MR4002167}.
	
	The main difference between the asymptotic analysis for the critical and subcritical regimes occurs because of the change in the monotonicity properties of the Pohozaev functional, which in classifies the type of stability for singular solutions to \eqref{ourPDE} around a blow-up (shrink-down) limit solution. 
	This method is inspired by Fleming's tangent cone analysis for minimal hypersurfaces \cite{MR157263,MR3190428}.
	In the critical case, it can be showed that the Pohozaev functional becomes constant, that is, blow-up limit solutions are stable.
	In contrast, in the subcritical case, they are asymptotically stable. 
	This discrepancy is caused by the sign-changing behavior of the coefficients of the tri-Laplacian in Emden--Fowler coordinates (or logarithm cylindrical coordinates), which are suitable for this problem (see Remark~\ref{rmk:signoncoefficients}). 
	
	The proof of Theorem~\ref{maintheorem} is divided into two parts.
	First, we prove the asymptotic symmetry of singular solutions to \eqref{ourPDE} in the punctured ball. 
	Second, we use some ODE analysis and the monotonicity properties of the Pohozaev functional to study the asymptotic behavior for solutions on the cylinder.
	The strategy strongly relies on the growth of the subcritical nonlinearity, which, in the lower critical situation $p=2_{\#}$, turns out to be far different from the other regime.
	We should emphasize that the superharmonicity conditions are not necessary in the range $p\in(2_{\#},2^{\#}-1)$, which is not the case otherwise.
	Third, we define two homological-type invariants that satisfy suitable monotonicity properties. This can be used to study the local asymptotic behavior of positive singular solutions near the isolated singularity.
	We need to subdivide our approach with respect to the growth of the nonlinearity into three cases; namely, Serrin--Lions case, Aviles case, and Gidas--Spruck case.

	We remark that in a companion paper the same authors study the asymptotics for positive singular solutions in the other limit situation $p=2^{\#}-1$.
	Indeed, together with \cite{MR1679783} and Theorems~\ref{maintheorem} and \ref{thm:luo-wei}, this would provide a holistic picture of the qualitative behavior of solutions to \eqref{ourPDE} in the broader range $p\in(1,2^{\#}-1]$.
	These results are inspired by the classical literature for semi-linear second order equations due to J. Serrin \cite{MR0170096}, P.-L. Lions \cite{MR605060}, P. Aviles \cite{MR875297}, B. Gidas and J. Spruck \cite{MR615628}, and L. A. Caffarelli et al. \cite{MR982351} with an improvement given by N. Korevaar et al. \cite{MR1666838}.
	For more results on asymptotic analysis, we refer the interested reader to  \cite{arxiv:1901.01678}.

	We should observe that \eqref{ourPDE}, \eqref{ourPDE4th}, and \eqref{ourPDE2th} are particular cases of a more general class of equations, which we describe as follows.
	More precisely, considering (classical) positive singular  solutions $u\in C^{2m}(\mathbb{R}^n\setminus\{0\})$ with $n\geqslant 2m:=N$ to the following family of subcritical even order poly-harmonic PDEs 
	\begin{flalign}\tag{$\mathcal P_{N,p,R}$}\label{ourhigherPDE}
		(-\Delta)^{m}u=f_p(u) \quad {\rm in} \quad B^*_R.
	\end{flalign}
	Here $B_R\setminus\{0\}\subset\mathbb R^n$ is the punctured ball of radius $R>0$, $\Delta^m=\Delta\circ\dots\circ\Delta$ is the poly-Laplacian, and the nonlinearity $f_p\in C^{1}(B_R)$ is given by
	\begin{equation*}
		f_p(u):=|u|^{p-1}u \quad {\rm with} \quad p\in\left[\frac{n}{n-2m},\frac{n+2m}{n-2m}\right):=[2_{m,*},2_m^{*}),
	\end{equation*}
	where $2_{m,*}:=\frac{n}{n-2m}$ and $2^m_{*}:=\frac{2n}{n-2m}$ are, respectively, the lower and upper critical exponents with respect to the compact Sobolev embedding of $H^{m}(\mathbb R^n)$.
	In this regard, it is natural to expect a classification result for the local asymptotic behavior of singular solutions to \eqref{ourhigherPDE} near isolated singularities in the sense of Theorem~\ref{maintheorem}.
	
	Let us state as the following conjecture 
	\begin{conjecture}\label{conj:poly-laplacian}
		Let $u$ be a positive singular solution to \eqref{ourhigherPDE} with $p\in(1,2_m^*-1)$. 
		Assume that $(-\Delta u)^j\geqslant0$ for any $j=1,\dots,m-1$.
		Then, it follows
		\begin{equation*}
			u(x)=(1+\mathcal{O}(|x|))\overline{u}(x) \quad {\rm as} \quad x\rightarrow0,
		\end{equation*}
		Moreover, 
		\begin{itemize}
			\item[{\rm (a)}] if $p\in(1,2^*_{m})$, then
			\begin{equation*}
				u(x)\simeq|x|^{N-n} \quad {\rm as} \quad x\rightarrow0;
			\end{equation*}
			\item[{\rm (b)}] if $p=2_{m,*}$, then 
			\begin{equation*}
				u(x)=(1+\mathrm o(1))\widehat{K}_{0}(n)^{\frac{n-N}{N}}|x|^{N-n}(\ln|x|)^{\frac{N-n}{N}} \quad {\rm as} \quad x\rightarrow0,
			\end{equation*}
			where
			\begin{equation*}
				\widehat{K}_{N,0}(n)=\frac{2^{m-2}(m-1)!}{m}\prod_{j=0}^{m-1}\left(n-2j\right)(n-N)^2.
			\end{equation*}
			\item[{\rm (c)}] if $p\in(2_{m,*},2_m^{*}-1)$, then
			\begin{equation*}
				u(x)=(1+\mathrm{o}(1))K_0(n,p)^{\frac{1}{p-1}}|x|^{-\frac{N}{p-1}} \quad {\rm as} \quad x\rightarrow0.
			\end{equation*}
		\end{itemize}
	\end{conjecture}
	
	\begin{remark}
		In a recent paper, it was independently proved by X. Huang, Y. Li, and H. Yang \cite{arXiv:2210.04619} that the super poly-harmonic condition can be removed.
		We also refer to Q. Ng\^o and D. Ye \cite{MR4438901} for more details on the blow-up limit case $R=+\infty$.
		We keep this condition in our manuscript because of its natural relations with curvature sign conditions for the upper critical situation $p=2^{\#}-1$.
		They also provided some upper bound estimates that will be important in our methods.
	\end{remark}
	
	Here is our plan for the rest of the paper. 
	In Section~\ref{sec:changeofvariables}, we introduce both the autonomous and nonautonomous Emden--Fowler coordinates. 
	In Section~\ref{sec:pohozaevfunctional}, we define the associated Pohozaev functionals, and we prove their (asymptotic) monotonicity properties.
	In Section~\ref{sec:upperbounds}, we prove some {\it a priori} upper bound estimates.
	In Section~\ref{sec:asympradialsymmetry}, we perform a variant of the integral moving spheres method and prove that solutions are asymptotically radially symmetric.
	In Section~\ref{sec:limitinglevels}, we study the limit values of the Pohozaev functional under blow-up and shrink-down sequences.
	In Section~\ref{sec:localbehavior}, we use the monotonicity formulas and some asymptotic analysis to prove the classification of the local asymptotic behavior in Theorem~\ref{maintheorem}.
	
	\numberwithin{equation}{section} 
	\numberwithin{theorem}{section}
	
	\section{Emden--Fowler coordinates}\label{sec:changeofvariables}
	In this section, we define the Emden--Fowler change of variables, which is used to transform the singular PDE \eqref{ourPDE} problem into an ODE problem.
	
	\subsection{Autonomous case}
	We define a classical change of variables, which transforms \eqref{ourPDE} into an ODE with constant coefficients.
	
	\begin{definition}
		Let us define the sixth order autonomous Emden--Fowler change of variables $($or cylindrical logarithm coordinates$)$ given by
		\begin{equation*}
			v(t,\theta)=r^{\gamma_p}u(r,\sigma), \quad \mbox{where} \quad t=\ln r, \quad \sigma=\theta=x|x|^{-1}.
		\end{equation*}  
	\end{definition}
	Using this coordinate system and performing a lengthy computation, we arrive at the following sixth order nonlinear PDE on the cylinder ${\mathcal{C}}_T:=(-\infty,T)\times \mathbb{S}^{n-1}$ with $T=\ln R<+\infty$,
	\begin{equation}\tag{$\mathcal C_{p,T}$}\label{ourPDEcyl}
		-P^3_{\rm cyl}v=f_p(v) \quad {\rm on} \quad {\mathcal{C}}_T.
	\end{equation}
	Here $P^3_{\rm cyl}$ is the tri-Laplacian written in cylindrical coordinates given by
	\begin{align*}
		P^3_{\rm cyl}&=\partial_t^{(6)}+K_{5}(n,p)\partial_t^{(5)}+K_{4}(n,p)\partial_t^{(4)}+K_{3}(n,p)\partial_t^{(3)}+K_{2}(n,p)\partial_t^{(2)}+K_{1}(n,p)\partial_t+K_{0}(n,p)&\\\nonumber
		&+2\partial^{(4)}_t\Delta_{\theta}+J_3(n,p)\partial^{(3)}_t\Delta_{\theta}+J_2(n,p)\partial^{(2)}_t\Delta_{\theta}+J_1(n,p)\partial_t\Delta_{\theta}+J_0(n,p)\Delta_{\theta}\\
		&+3\partial^{(2)}_t\Delta^2_{\theta}+L_1(n,p)\partial_t\Delta^2_{\theta}+L_0(n,p)\Delta^2_{\theta}+\Delta^3_{\theta},&
	\end{align*}
	where $K_j(n,p), J_j(n,p), L_j(n,p)$ for $j=0,1,2,3,4,5$ are dimensional constants given by \eqref{autonomouscoefficients1} and \eqref{autonomouscoefficients2}.
	
	\begin{definition}
		Let us consider the autonomous Emden--Fowler transformation as follows
		\begin{equation}\label{cyltransform}
			\mathfrak{F}:C_c^{\infty}(B_R^*)\rightarrow C_c^{\infty}(\mathcal{C}_{T}) \quad \mbox{given by} \quad
			\mathfrak{F}(u)=e^{\gamma_pt}u(e^{t},\theta):=v.
		\end{equation}
	\end{definition}
	
	\subsection{Non-autonomous case}
	In the lower critical case, because of the vanishing of the coefficient $K_0(n,2_{\#})$, we already know that the situation changes dramatically. 
	In this fashion, we define the so-called {nonautonomous Emden--Fowler change of variables} \cite{MR875297,MR4085120,MR633393,MR727703}.
	
	\begin{definition}
		Let us define the sixth order nonautonomous Emden--Fowler change of variables $($or cylindrical logarithm coordinates$)$ given by
		\begin{equation}\label{newcylindrical}
			w(t,\theta)=r^{6-n}(\ln r)^{\frac{6-n}{6}}u(r,\sigma), \quad {\rm where} \quad t=\ln r \quad {\rm and} \quad \sigma=\theta=x|x|^{-1}.
		\end{equation}
	\end{definition}
	Using this coordinate system and performing a lengthy computation, we arrive at 
	\begin{equation}\tag{$\widetilde{\mathcal C}_{T}$}\label{ourPDEcylnon}
		-\widetilde{P}^3_{\rm cyl}w=t^{-1}|w|^{2_{\#}-1}w \quad {\rm on} \quad {\mathcal{C}}_T.
	\end{equation}
	Here $\widetilde{P}^3_{\rm cyl}$ is the tri-Laplacian written in nonautonomous Emden--Fowler coordinates given by
	\begin{align*}
		\widetilde{P}^3_{\rm cyl}&=\partial_t^{(6)}+\widetilde{K}_{5}(n,t)\partial_t^{(5)}+\widetilde{K}_{4}(n,t)\partial_t^{(4)}+\widetilde{K}_{3}(n,t)\partial_t^{(3)}+\widetilde{K}_{2}(n,t)\partial_t^{(2)}+\widetilde{K}_{1}(n,p)\partial_t+\widetilde{K}_{0}(n,p)&\\\nonumber
		&+\widetilde{J}_4(n,t)\partial^{(4)}_t\Delta_{\theta}+\widetilde{J}_3(n,t)\partial^{(3)}_t\Delta_{\theta}+\widetilde{J}_2(n,t)\partial^{(2)}_t\Delta_{\theta}+\widetilde{J}_1(n,t)\partial_t\Delta_{\theta}+\widetilde{J}_0(n,t)\Delta_{\theta}\\
		&+\widetilde{L}_2(n,t)\partial^{(2)}_t\Delta^2_{\theta}+\widetilde{L}_1(n,t)\partial_t\Delta^2_{\theta}+\widetilde{L}_0(n,t)\Delta^2_{\theta}+\Delta^3_{\theta},&
	\end{align*}
	where $\widetilde{K}_j(n,t), \widetilde{J}_j(n,t), \widetilde{L}_j(n,t)$ for $j=0,1,2,3,4,5$ are real-valued functions given by \eqref{coeficcientnonautonomous1} and \eqref{coeficcientnonautonomous2}. 
	
	Analogously to the standard autonomous case, we also consider a transformation that sends a singular solution to \eqref{ourPDE} with $p=2_{\#}$ into solutions to a nonautonomous ODE on the cylinder.
	\begin{definition}
		Let us consider the non-autonomous Emden--Fowler transformation as follows
		\begin{equation}\label{noncyltransform}
			\tilde{\mathfrak{F}}:C_c^{\infty}(B_R^*)\rightarrow C_c^{\infty}(\mathcal{C}_T) \quad \mbox{given by} \quad
			\tilde{\mathfrak{F}}(u)=e^{(6-n)t}t^{\frac{6-n}{6}}u(e^{t},\theta):=w.		
		\end{equation}
	\end{definition}
	
	\section{Pohozaev functionals}\label{sec:pohozaevfunctional}
	Next, we define two central characters in our studies, the so-called autonomous and nonautonomous Pohozaev functionals. 
	The heuristics are that these Pohozaev functionals classify whether or not the blow-up limit solutions in cylindrical coordinates are (asymptotically) stable equilibrium points of the associated sixth order ODE. 
	More precisely, from the dynamical systems point of view, it works as a Lyapunov functional.
	Let us remark that to prove the Pohozaev invariant to be well-defined, one needs to use some upper estimates and asymptotic symmetry that will be proved independently in Section~\ref{sec:asympradialsymmetry} (see Lemmas~\ref{lm:uniformestimategidasspruck} and \ref{lm:sharpestimate}).
	
	\subsection{Autonomous case}
	Initially, let us define {the Pohozaev functional} associated to the PDE equation on the cylinder given by \eqref{ourPDEcyl}, which will play a central role in our analysis (for more details on this class of invariants, we refer to \cite{MR3869387,MR4094467,MR4123335,arXiv:1804.00817}).
	
	\begin{remark}\label{rmk:signoncoefficients}
		By direct computations, notice that when $p\in(2_{\#},2^{\#}-1)$, one has the sign relations 
		\begin{equation*}
			K_5(n,p),K_1(n,p),L_1(n,p)\geqslant 0, \quad  {\rm and} \quad K_3(n,p),J_3(n,p),J_1(n,p)\leqslant 0.
		\end{equation*}
		In addition, we can explicitly compute these coefficients in the limiting situations $p=2_{\#}$ and $p=2^{\#}-1$ \eqref{coeficcientslower} and \eqref{coeficcientsupper}.
		This change in sign behavior explains why one needs to use distinct methods when the power parameter $p\in(1,+\infty)$ changes.
		
		With respect to Conjecture~\ref{conj:poly-laplacian}, we speculate that a relationship like this should hold for the coefficients of the poly-Laplacian written in Emden--Fowler coordinates.
	\end{remark}

	\begin{definition}
		For any $v\in C^6(\mathbb R)$ positive solution to \eqref{ourPDEcyl} with $p\in(2_{\#},2^{\#}-1)$, let us define its {\it cylindrical Pohozaev functional} by 
		\begin{align*}
			\mathcal{P}_{\rm cyl}(t,v)&=\displaystyle\int_{\mathbb{S}_t^{n-1}}\mathcal{H}(t,\theta,v)\ud\theta.&
		\end{align*}
		Here $\mathbb{S}_t^{n-1}=\{t\}\times\mathbb{S}^{n-1}$ is the cylindrical ball with volume element given by $\ud\theta=e^{-2t}\ud\sigma$, where $\ud\sigma_r$ is the volume element of the ball of radius $r>0$ in $\mathbb{R}^n$. 
		In addition, we set
		\begin{align*}
			\mathcal{H}(t,\theta,v):=\mathcal{H}_{{\rm rad}}(t,\theta,v)+\mathcal{H}_{\rm ang}(t,\theta,v),
		\end{align*}
		where 
		\begin{align*}
			\mathcal{H}_{{\rm rad}}(t,\theta,v)&:=\left(v^{(5)}v^{(1)}-v^{(4)}v^{(2)}+\frac{1}{2}{v^{(3)}}^2\right)+K_5\left(v^{(4)}v^{(1)}-v^{(3)}v^{(2)}\right)&\\
			&+{K_4}\left(v^{(3)}v^{(1)}-\frac{1}{2}{v^{(2)}}^{2}\right)+K_3v^{(2)}v^{(1)}+\frac{K_2}{2}{v^{(1)}}^2+\frac{K_0}{2}v^2+\frac{|v|^{p+1}}{p+1}
		\end{align*}
		and
		\begin{align*}
			\mathcal{H}_{\rm ang}(t,\theta,v)
			&:=-J_4\left(\partial_t^{(3)}\nabla_\theta v\partial_t\nabla_\theta v-|\partial_t^{(2)}\nabla_\theta v|^2\right)-\frac{J_2}{2}|\partial_t^{(2)}\nabla_\theta v|^2-\frac{J_1}{2}|\partial_t^{(2)}\nabla_\theta v|^2-\frac{J_0}{2}|\nabla_\theta v|^2&\\
			&+\frac{L_2}{2}|\partial_t^{(2)}\Delta_\theta v|^2+\frac{L_0}{2}|\partial_t^{(2)}\Delta_\theta v|^2+\frac{1}{2}|\Delta_{\theta}v|^2.&		
		\end{align*}
	\end{definition}
	
	\begin{remark}\label{rmk:pohozaevspherical}
		Using the inverse of the Emden--Fowler transformation, one can also construct the {\it spherical Pohozaev functional} given by $\mathcal{P}_{\rm sph}(r,u):=\left(\mathcal{P}_{\rm cyl}\circ\mathfrak{F}^{-1}\right)\left(t,v\right)$. 
		From \cite[Propositon~A.1]{arxiv:1901.01678}, it follows that $\mathcal{P}_{\rm sph}$ also satisfies a monotonicity property.
	\end{remark}
	
	We deduce a monotonicity formula for the logarithmic cylindrical Pohozaev functional $\mathcal{P}_{\rm cyl}(t,v)$, which will be essential to show that the limit Pohozaev invariant is well-defined as $t\rightarrow-\infty$.
	\begin{proposition}\label{lm:monotonicityformula}
		Let $v\in C^6(\mathbb R)$ be a positive solution to \eqref{ourPDEcyl} with $p\in(2_{\#},2^{\#}-1)$.
		If $-\infty<t_1\leqslant t_2<T$, then $\mathcal{P}_{\rm cyl}(t_2,v)-\mathcal{P}_{\rm cyl}(t_1,v)\leqslant0$. 
		More precisely, we have the monotonicity formula
		\begin{align}\label{eq:monotonicityformula}
			\frac{\ud}{\ud t}\mathcal{P}_{\rm cyl}(t,v)=\int_{\mathbb{S}_t^{n-1}}\left(-K_{5}|\partial_t^{(3)}v|^{2}+K_{3}|\partial_t^{(2)}v|^{2}-K_{1}|\partial_tv|^{2}-J_{3}|\partial^{(2)}_t\nabla_{\theta}v|^2-L_{1}|\partial_t\Delta_{\theta}v|^2\right)\ud\theta.
		\end{align}
		In particular, ${\mathcal{P}}_{\rm cyl}(t,v)$ is nonincreasing, and so $\mathcal{P}_{\rm cyl}(-\infty,v):=\lim_{t\rightarrow-\infty}{\mathcal{P}}_{\rm cyl}(t,v)$ exists.
	\end{proposition}
	
	\begin{proof}
		Initially, observe that \eqref{eq:monotonicityformula} follows by a direct computation, which consists multiplying  \eqref{ourPDEcyl} by $v^{(1)}$, and integrating by parts. 
		Hence, using Remark~\ref{rmk:signoncoefficients}, we find $\partial_t\mathcal{P}_{\rm cyl}(t,v)\leqslant0$, which proves that the cylindrical Pohozaev functional is nonincreasing. 
		Finally, since by Lemma~\ref{lm:universalestimates}, the Pohozaev functional is bounded above, the full existence of limit follows.
	\end{proof}
	
	\begin{remark}\label{rmk:limitrelation}
		Using the inverse of cylindrical transform, that is,  $\mathcal{P}_{\rm sph}=\mathcal{P}_{\rm cyl}\circ\mathfrak{F}^{-1}$, it follows that $\mathcal{P}_{\rm sph}(0,u)=\lim_{r\rightarrow0}\mathcal{P}_{\rm sph}(r,u)=\lim_{t\rightarrow-\infty}{\mathcal{P}}_{\rm cyl}(t,v).$
		The last equality implies that the Pohozaev invariant is well-defined in the punctured ball when $R\rightarrow0$.
	\end{remark}
	
	\subsection{Non-autonomous case}
	In the lower critical case $p=2_{\#}$, since $K_{0,\#}:=K_0(n,2_{\#})=0$ and $\gamma_{2_{\#}}=n-6$, a new cylindrical transformation was defined.
	Concerning this nonautonomous cylindrical transformation, we compute its associated Pohozaev functional, which in this situation has some time-dependent terms.
	
	\begin{definition} 
		For any $w\in C^6(\mathbb R)$ be a positive solution to \eqref{ourPDEcylnon}, let us define its {\it cylindrical nonautonomous Pohozaev functional} by 
		\begin{align*}
			\widetilde{\mathcal{P}}_{\rm cyl}(t,w)&=\displaystyle\int_{\mathbb{S}_t^{n-1}}\widetilde{\mathcal{H}}(t,\theta,w)\ud\theta.&
		\end{align*} 
		Here
		\begin{align*}
			\widetilde{\mathcal{H}}(t,\theta,w):=\widetilde{\mathcal{H}}_{{\rm rad}}(t,\theta,w)+\widetilde{\mathcal{H}}_{\rm ang}(t,\theta,w),
		\end{align*}
		where
		\begin{align*}
			\widetilde{\mathcal{H}}_{{\rm rad}}(t,\theta,w)&:=t\left(w^{(5)}w^{(1)}-w^{(4)}w^{(2)}+\frac{1}{2}{w^{(3)}}^2\right)+t\widetilde K_5\left(w^{(4)}w^{(1)}-w^{(3)}w^{(2)}\right)&\\
			&+t\widetilde{K}_4\left(w^{(3)}w^{(1)}-\frac{1}{2}{w^{(2)}}^{2}\right)+t\widetilde K_3w^{(2)}w^{(1)}+\frac{t\widetilde K_1}{2}{w^{(1)}}^2+\frac{t \widetilde K_0}{2}w^2+\frac{n-6}{2(n-3)}|w|^{2_{\#}+1}
		\end{align*}
		and
		\begin{align*}
			\widetilde{\mathcal{H}}_{\rm ang}(t,\theta,w)&:=-t\widetilde{J}_4\left(\partial_t^{(3)}\nabla_\theta w\partial_t\nabla_\theta w-|\partial_t^{(2)}\nabla_\theta w|^2\right)-\frac{t\widetilde{J}_2}{2}|\partial_t^{(2)}\nabla_\theta w|^2-\frac{t\widetilde{J}_1}{2}|\partial_t^{(2)}\nabla_\theta w|^2-\frac{t\widetilde{J}_0}{2}|\nabla_\theta w|^2&\\
			&+\frac{t\widetilde{L}_2}{2}|\partial_t^{(2)}\Delta_\theta w|^2+\frac{t\widetilde{L}_0}{2}|\partial_t^{(2)}\Delta_\theta w|^2+\frac{t}{2}|\Delta_{\theta}w|^2.&
		\end{align*}
	\end{definition}
	
	\begin{remark}\label{rmk:estimates}
		Initially, notice that the nonautonomous Pohozaev invariant is well-defined for any $t\in(-\infty,T)$ since $u\in C^{6}(\mathbb R^n\setminus\{0\})$ is smooth away from the origin. 
		In addition, due to Proposition~\ref{prop:asymptoticsymmetry}, for $R<\infty$, we get that $|w(t,\theta)-\overline{w}(t)|=\mathcal{O}(e^{\beta t})$ as $t\rightarrow-\infty $, for some $\beta>0$, where $\overline{w}(t)$ is the average of $w(t,\theta)$ over $\theta\in \mathbb{S}_t^{n-1}$ with $w=\widetilde{\mathfrak{F}}(u)$ given by \eqref{noncyltransform}. In particular, one can find some large $T_{0}\gg1$ and $C>0$ $($independent of $t$$)$ satisfying
		\begin{equation}\label{angularestimates}
			\left|\nabla_{\theta} w(t,\theta)\right|+\left|\Delta_{\theta} w(t,\theta)\right|+\left|\nabla_{\theta}\Delta_{\theta} w(t,\theta)\right|+\left|\Delta^2_{\theta} w(t,\theta)\right| \leqslant Ce^{\beta t} \quad {\rm in} \quad \mathcal{C}_{T_0}.
		\end{equation}
		Moreover, from the sharp estimate in Proposition~\ref{lm:monotonicityformula} and the gradient estimate \eqref{lm:gradientestimates}, it follows 
		\begin{equation}\label{radialestimates}
			\sum_{j=0}^5|\partial_t^{(j)}w(t,\theta)|\leqslant C \quad {\rm in} \quad \mathcal{C}_{T_0}.
		\end{equation}
	\end{remark}
	
	Before we prove the monotonicity formula, we need to show an auxiliary result 
	
	\begin{lemma}\label{lm:signproperties}
		Let $w\in C^6(\mathbb R)$ be a positive solution to \eqref{ourPDEcylnon}. 
		Then, $\lim_{t\rightarrow-\infty}w(t,\theta)\in \{0,\widehat{K}_{0}(n)^{\frac{n-6}{6}}\}$, 
		where $\widehat{K}_{j}(n)=\lim_{t\rightarrow-\infty}t\widetilde{K}_j(n,t)$ for $j=0,1,2,3,4,5$.
		Moreover,
		\begin{equation*}
			\lim_{t\rightarrow-\infty}{w}^{(j)}(t,\theta)=\lim_{t\rightarrow-\infty}\overline{w}^{(j)}(t)=0 \quad {\rm for \ all} \quad j\geqslant1.
		\end{equation*}
	\end{lemma} 
	
	\begin{proof}
		Indeed, using the asymptotic symmetry in Proposition~\ref{prop:asymptoticsymmetry} and the estimates in Remark~\ref{rmk:estimates}, we have that the cylindrical transformation of spherical average $\overline{w}=\widetilde{\mathfrak{F}}(\overline{u})$ satisfies,
		\begin{equation*}
			\overline{P}_{\rm cyl}\overline{w}=-|\overline{w}|^{2_{\#}-1}\overline{w}+\mathcal{O}(\overline{w}(t)e^{t}) \quad {\rm as} \quad t\rightarrow-\infty,
		\end{equation*} 
		where
		\begin{equation*}
			\overline{P}_{\rm cyl}=t\partial_t^{(6)}+\widehat{K}_{5}(n)\partial_t^{(5)}+\widehat{K}_{4}(n)\partial_t^{(4)}+\widehat{K}_{3}(n)\partial_t^{(3)}+\widehat{K}_{2}(n)\partial_t^{(2)}+\widehat{K}_{1}(n)\partial_t+\widehat{K}_{0}(n).
		\end{equation*} 
		Thus, decomposing the left-hand side of the last equation into a product of three second order operators, we can adapt the arguments in \cite[Section~2]{MR4094467} (see also \cite[Section~4]{arXiv:2002.12491})
		for the last sixth order asymptotic identity. In this fashion, we split the proof of the claim into two steps:
		
		\noindent{\bf Step 1:} Either $w(t,\theta)=\mathrm{o}(1)$ or $w(t,\theta)=\widehat{K}_0(n)^{\frac{n-6}{6}}+\mathrm{o}(1)$ as $t\rightarrow-\infty $.
		
		\noindent As a matter of fact, the conclusion follows directly from the same argument in \cite[Lemma~4.10]{arXiv:2210.04376} to the limit ODE \eqref{ourPDEcylnon} with $T=\infty$.
		
		\noindent{\bf Step 2:} $\partial_t^{(j)}w(t,\theta)=\mathrm{o}(1)$ as $t\rightarrow-\infty $ for all $j\geqslant1$.
		
		\noindent Indeed, we can adapt the proof in \cite[Lemma~4.8]{arXiv:2210.04376} to the ODE \eqref{ourPDEcylnon} with $T=\infty$.	
		
		Steps 1 and 2 together give the proof for the first claim.
	\end{proof}
	
	To prove the full existence of $\widetilde{\mathcal{P}}_{\rm cyl}(-\infty,w)$, we shall verify the estimates in the next lemma, which is a sixth order version of \cite[Lemma~3.2]{MR875297}. 
	Due to the appearance of higher order derivative terms, we give different proof to the ones in the lemma quoted above. 
	Namely, our proof is based on Lemma~\ref{lm:signproperties} combined with a simple L'H\^ospital rule.
	Initially, by differentiating $\widetilde{\mathcal{H}}(t,\theta,w)$ with respect to $t$, integrating by parts over $\mathbb{S}_t^{n-1}$ (using differentiation under the integral sign one can even omit the dependence on $t$), and using \eqref{ourPDEcylnon}, we find
	\begin{align*}
		\frac{\ud}{\ud t}\widetilde{\mathcal{P}}_{\rm cyl}(t,w)
		=\Xi_{\rm rad}(t,w)+\Xi_{\rm ang}(t,w).
	\end{align*}
	Here 
	\begin{align}\label{angularerror}
		\quad \Xi_{\rm rad}(t,w):=\int_{\mathbb{S}^{n-1}}\partial_t\widetilde{\mathcal{H}}_{{\rm rad}}(t,\theta,w)\ud\theta \quad {\rm and} \quad \Xi_{\rm ang}(t,w):=-\int_{\mathbb{S}^{n-1}}\partial_t\widetilde{\mathcal{H}}_{{\rm ang}}(t,\theta,w)\ud\theta,
	\end{align}
	where
	\begin{align*}
		\partial_t\widetilde{\mathcal{H}}_{{\rm rad}}(t,\theta,w)&=w^{(5)}w^{(1)}-w^{(4)}w^{(3)}+{w^{(3)}}^2+\mathfrak{p}_5w^{(4)}w^{(1)}-\mathfrak{p}_5w^{(3)}w^{(2)}&\\\nonumber
		&+ \mathfrak{p}_4w^{(3)}w^{(1)}-\mathfrak{p}_4{w^{(2)}}^2+\mathfrak{p}_3w^{(2)}w^{(1)}+\mathfrak{p}_2{w^{(2)}}^{2}+\mathfrak{p}_1{w^{(1)}}^2+\mathfrak{p}_0w^2.&
	\end{align*}
	The real-valued functions $\mathfrak{p}_j(n,\cdot):(-\infty,\ln R)\rightarrow\mathbb R$ are given by \eqref{coefficientspohozaev}.
	More explicitly, we have
	\begin{align}\label{estimate0}
		\mathfrak{p}_0(n,t)&:=\frac{1}{36}(3n^3-48n^2+228n-320)t^{-2}+\mathcal{O}(t^{-4}) \quad {\rm as} \quad t\rightarrow-\infty ,&\\\label{estimate1}
		\mathfrak{p}_1(n,t)&:=\frac{1}{864}(n-6)(3456n^2-20736n+27648)t&\\\nonumber
		&-\frac{1}{864}(n-6)(864n^3-10368n^2+44928n-64512)+\mathcal{O}(t^{-1}) \quad {\rm as} \quad t\rightarrow-\infty ,&\\\label{estimate2}
		\mathfrak{p}_2(n,t)&:=-(3n^3-48n^2+228n-320)+\mathcal{O}(t^{-2}) \quad {\rm as} \quad t\rightarrow-\infty ,&\\\label{estimate3}
		\mathfrak{p}_3(n,t)&:=-\frac{1}{2}(n^3-30n^2+212n-408)+\mathcal{O}(t^{-2}) \quad {\rm as} \quad t\rightarrow-\infty ,&\\\label{estimate4}
		\mathfrak{p}_4(n,t)&:=\frac{1}{2}(3 n^2-42n+124)+\mathcal{O}(t^{-2}) \quad {\rm as} \quad t\rightarrow-\infty ,&\\\label{estimate5}
		\mathfrak{p}_5(n,t)&:=-\frac{1}{2}{(n-6)}.&
	\end{align}
	
	Now we use the asymptotic estimate proved in Lemma~\ref{lm:signproperties} to prove the result below
	
	\begin{lemma}\label{lm:estimateangularparts}
		Let $w\in C^6(\mathbb R)$ be a positive solution to \eqref{ourPDEcylnon}. Then, $\lim_{t\rightarrow-\infty }\Xi_{\rm rad}(t,w)=0.$
	\end{lemma}
	
	\begin{proof}
		Notice that 
		\begin{align*}
			\Xi_{\rm rad}(t,w)&:=I_{51}+I_{42}+I_{33}+I_{41}+I_{32}+I_{31}+I_{22}+I_{21}+I_{22}+I_{11}+I_{00}.
		\end{align*}
		Next, we estimate each term of the last identity separately by steps. 
		
		\noindent{\bf Step 1:}
		$I_{00}:=\lim_{t\rightarrow-\infty }\int_{\mathbb{S}^{n-1}}\mathfrak{p}_0(n,t)w(t,\theta)^2\ud\theta=0$.
		
		\noindent Using \eqref{estimate0}, we find that $\lim_{t\rightarrow-\infty }\mathfrak{p}_0(n,t)=0$, which yields
		\begin{align*}
			\lim_{t\rightarrow-\infty }\int_{\mathbb{S}^{n-1}}\mathfrak{p}_0(n,t)w^{(1)}(t,\theta)^2\ud\theta=\left(\lim_{t\rightarrow-\infty }\mathfrak{p}_0(n,t)\right)\left(\lim_{t\rightarrow-\infty }\int_{\mathbb{S}^{n-1}}w^{(1)}(t,\theta)^2\ud\theta\right)=0.&
		\end{align*}
		This gives us the desired conclusion.
		
		\noindent{\bf Step 2:} $I_{11}:=\lim_{t\rightarrow-\infty }\int_{\mathbb{S}^{n-1}}\mathfrak{p}_1(n,t)w^{(1)}(t,\theta)^2\ud\theta=0$.
		
		\noindent The difference in this case is $\lim_{t\rightarrow-\infty }\mathfrak{p}_1(n,t)=-\infty$.
		However, using \eqref{estimate1}, we can decompose $\mathfrak{p}_1(n,t)=\widetilde{\mathfrak{p}}_1(n,t)+\widehat{\mathfrak{p}}_1(n,t)$, where $\lim_{t\rightarrow-\infty }\widetilde{\mathfrak{p}}_1(n,t)<+\infty $ and $\widehat{\mathfrak{p}}_1(n,t)=\frac{1}{864}(n-6)(3456n^2-20736n+27648)t$.
		With this notation, we set
		\begin{align*}
			\lim_{t\rightarrow-\infty }\int_{\mathbb{S}^{n-1}}\mathfrak{p}_1(n,t)w^{(1)}(t,\theta)^2\ud\theta
			:=\widetilde{I}_{11}+\widehat{I}_{11}.
		\end{align*}
		Notice that since $\lim_{t\rightarrow-\infty }\widetilde{\mathfrak{p}}_1(n,t)=0$, we trivially see that $\widetilde{I}_{11}=0$. Moreover, in order to estimate $\widehat{I}_{11}$, we use the L'H\^ospital rule as follows $\widehat{I}_{11}=\lim_{t\rightarrow-\infty }\int_{\mathbb{S}^{n-1}}\widehat{\mathfrak{p}}_1(n,t)w(t,\theta)^2\ud\theta=0,$	which proves this step.
		
		\noindent{\bf Step 3:}
		$I_{22}:=\lim_{t\rightarrow-\infty }\int_{\mathbb{S}^{n-1}}\mathfrak{p}_2(n,t)w^{(2)}(t,\theta)^{2}\ud\theta=0.$
		
		\noindent As before, by \eqref{estimate2}, it holds that $\lim_{t\rightarrow-\infty}\mathfrak{p}_2(n,t)<\infty$ and using
		\begin{align*}
			\lim_{t\rightarrow\infty}\int_{\mathbb{S}^{n-1}}\mathfrak{p}_2(n,t)w^{(2)}(t,\theta)^2\ud\theta=\left(\lim_{t\rightarrow-\infty}\mathfrak{p}_2(n,t)\right)\left(\lim_{t\rightarrow-\infty}\int_{\mathbb{S}^{n-1}}w^{(2)}(t,\theta)^2\ud\theta\right)=0,&
		\end{align*}
		the proof of this step follows promptly.
		
		\noindent{\bf Step 4:} $I_{21}:=\lim_{t\rightarrow-\infty }\int_{\mathbb{S}^{n-1}}\mathfrak{p}_3(n,t)w^{(2)}(t,\theta)w^{(1)}(t,\theta)\ud\theta=0$.
		
		\noindent In fact, using \eqref{estimate3}, one has that  $\lim_{t\rightarrow-\infty }\mathfrak{p}_3(n,t)<+\infty$. Thus, 
		\begin{align*}
			\lim_{t\rightarrow-\infty }\int_{\mathbb{S}^{n-1}}\mathfrak{p}_3(n,t)w^{(2)}(t,\theta)w^{(1)}(t,\theta)\ud\theta=\left(\lim_{t\rightarrow-\infty }\mathfrak{p}_3(n,t)\right)\left(\lim_{t\rightarrow-\infty }\int_{\mathbb{S}^{n-1}}w^{(2)}(t,\theta)w^{(1)}(t,\theta)\ud\theta\right)=0,&
		\end{align*}
		which by Lemma~\ref{lm:signproperties}
		leads to the conclusion.
		
		\noindent{\bf Step 5:} $I_{22}:=-\lim_{t\rightarrow-\infty }\int_{\mathbb{S}^{n-1}}\mathfrak{p}_4(n,t)w^{(2)}(t,\theta)^2\ud\theta=0$.
		
		\noindent Indeed, by \eqref{estimate4}, we obtain that $\lim_{t\rightarrow-\infty }\mathfrak{p}_4(n,t)<+\infty$. Thus, 
		\begin{align*}
			-\lim_{t\rightarrow-\infty }\int_{\mathbb{S}^{n-1}}\mathfrak{p}_4(n,t)w^{(2)}(t,\theta)^2\ud\theta=-\left(\lim_{t\rightarrow-\infty }\mathfrak{p}_4(n,t)\right)\left(\lim_{t\rightarrow-\infty }\int_{\mathbb{S}^{n-1}}w^{(2)}(t,\theta)^2\ud\theta\right)=0,&
		\end{align*}
		which by Lemma~\ref{lm:signproperties} concludes the argument.
		
		\noindent{\bf Step 6:} $I_{31}:=\lim_{t\rightarrow-\infty }\int_{\mathbb{S}^{n-1}}\mathfrak{p}_4(n,t)w^{(3)}(t,\theta)w^{(1)}(t,\theta)\ud\theta=0$.
		
		\noindent Using the same argument as before, we find
		\begin{align*}
			\lim_{t\rightarrow-\infty }\int_{\mathbb{S}^{n-1}}\mathfrak{p}_4(n,t)w^{(3)}(t,\theta)w^{(1)}(t,\theta)\ud\theta=\left(\lim_{t\rightarrow-\infty }\mathfrak{p}_4(n,t)\right)\left(\lim_{t\rightarrow-\infty }\int_{\mathbb{S}^{n-1}}w^{(3)}(t,\theta)w^{(1)}(t,\theta)\ud\theta\right)=0,&
		\end{align*}
		which finishes this step.
		
		\noindent{\bf Step 7:} $I_{32}:=-\lim_{t\rightarrow-\infty }\int_{\mathbb{S}^{n-1}}\mathfrak{p}_5(n,t)w^{(3)}(t,\theta)w^{(2)}(t,\theta)\ud\theta=0$.
		
		\noindent Indeed, by \eqref{estimate5}, we obtain that $\lim_{t\rightarrow-\infty }\mathfrak{p}_4(n,t)<+\infty$. Thus, 
		\begin{align*}
			-\lim_{t\rightarrow-\infty }\int_{\mathbb{S}^{n-1}}\mathfrak{p}_5(n,t)w^{(3)}(t,\theta)w^{(2)}(t,\theta)\ud\theta=-\left(\lim_{t\rightarrow-\infty }\mathfrak{p}_5(n,t)\right)\left(\lim_{t\rightarrow-\infty }\int_{\mathbb{S}^{n-1}}w^{(3)}(t,\theta)w^{(2)}(t,\theta)\ud\theta\right)=0,&
		\end{align*}
		which by Lemma~\ref{lm:signproperties} concludes the argument.
		
		\noindent{\bf Step 8:} $I_{41}:=\lim_{t\rightarrow-\infty }\int_{\mathbb{S}^{n-1}}\mathfrak{p}_5(n,t)w^{(4)}(t,\theta)w^{(1)}(t,\theta)\ud\theta=0$.
		
		\noindent Using the same argument as before, we find
		\begin{align*}
			\lim_{t\rightarrow-\infty }\int_{\mathbb{S}^{n-1}}\mathfrak{p}_5(n,t)w^{(4)}(t,\theta)w^{(1)}(t,\theta)\ud\theta=\left(\lim_{t\rightarrow-\infty }\mathfrak{p}_5(n,t)\right)\left(\lim_{t\rightarrow-\infty }\int_{\mathbb{S}^{n-1}}w^{(4)}(t,\theta)w^{(1)}(t,\theta)\ud\theta\right)=0,&
		\end{align*}
		which finishes this step.
		
		\noindent{\bf Step 9:} $I_{33}:=\lim_{t\rightarrow-\infty }\int_{\mathbb{S}^{n-1}}w^{(3)}(t,\theta)^2\ud\theta=0$.
		
		\noindent It follows directly by Lemma~\ref{lm:signproperties}.
		
		\noindent{\bf Step 10:} $I_{42}:=\lim_{t\rightarrow-\infty }\int_{\mathbb{S}^{n-1}} w^{(4)}(t,\theta)w^{(2)}(t,\theta)\ud\theta=0$.
		
		\noindent It follows directly by Lemma~\ref{lm:signproperties}.
		
		\noindent{\bf Step 11:} $I_{51}:=\lim_{t\rightarrow-\infty }\int_{\mathbb{S}^{n-1}}w^{(5)}(t,\theta),w^{(1)}(t,\theta)\ud\theta=0$.
		
		\noindent It follows directly by Lemma~\ref{lm:signproperties}.
		
		Finally, putting together all the steps the proof is concluded.
	\end{proof}
	
	The next proposition is the monotonicity of this new Pohozaev functional.
	
	\begin{proposition}\label{lm:lowermonotonicity}
		Let $w\in C^6(\mathbb R)$ be a positive solution to \eqref{ourPDEcylnon}.
		Then,  one has 
		\begin{itemize}
			\item[{\rm (i)}] if $n=7,8$, then there exists $T_*\gg1$ such that $\widetilde{\mathcal{P}}_{\rm cyl}(t,w)$ is non-decreasing for $t>T_*$.
			\item[{\rm (ii)}] if $n\geqslant9$, then there exists $T^*\gg1$ such that $\widetilde{\mathcal{P}}_{\rm cyl}(t,w)$ is nonincreasing for $t>T^*$.\\
			Moreover, $\widetilde{\mathcal{P}}_{\rm cyl}(-\infty ,w):=\lim_{t\rightarrow-\infty }\widetilde{\mathcal{P}}_{\rm cyl}(t,w)$ exists.
		\end{itemize}
	\end{proposition}
	
	\begin{proof}
		Initially, using Remark~\ref{rmk:estimates} and Lemma~\ref{lm:signproperties} one can find $t_0\gg1$ such that $\operatorname{sign}(\partial_t\widetilde{\mathcal{H}}_{{\rm rad}}(t,w))=\operatorname{sign}(\mathfrak{p}_0(n,t)|w|^2)$ for $|t|>t_0$. 
		Besides, we directly verify for $n\geqslant9$ that there exists $t_1\gg1$ sufficiently large such that 
		$\mathfrak{p}_0(n,t)<0$ for $|t|>t_1$, which, by taking $T^*:=\max\{t_1,t_0\}$, implies that $\widetilde{\mathcal{P}}_{\rm cyl}(t,w)$ is nonincreasing for $|t|>T^*$.
		However, for $n=7,8$, there exists $t_1\gg1$ sufficiently large such that 
		$\mathfrak{p}_0(n,t)<0$ for $|t|>t_2$; thus, setting $T_*:=\max\{t_0,t_2\}$, we get that $\widetilde{\mathcal{P}}_{\rm cyl}(t,w)$ is nondecreasing for $|t|>T_*$.
		Hence, the existence of $\widetilde{\mathcal{P}}_{\rm cyl}(-\infty ,w)$ follows  since $\widetilde{\mathcal{P}}_{\rm cyl}(t,w)$ is uniformly bounded both from above and below as $t \rightarrow -\infty $, and, by \eqref{angularestimates}, it gives us $\lim_{t\rightarrow-\infty }\Xi_{\rm ang}(t,w)=0$,	where $\Xi_{\rm ang}$ is defined by \eqref{angularerror}.
		This combined with \eqref{radialestimates} yields that $\liminf_{t \rightarrow -\infty} \widetilde{\mathcal{P}}_{\rm cyl}(t,w)<+\infty$.
		The proof is concluded.
	\end{proof}
	
	\section{A priori upper bounds}\label{sec:upperbounds}
	This subsection is devoted to provide a priori upper bounds for positive singular solutions to \eqref{ourPDE} with $p\in(1,2^{\#}-1)$.
	Our strategy relies on the classification of the limit solutions to \eqref{ourPDE} combined with a blow-up argument. 
	When $p\in(2_{\#},2^{\#}-1)$, most of our arguments in this section are similar in spirit to the ones in \cite{arXiv:2210.04619}, so we omit them here.
	Nevertheless, for $p=2_{\#}$, our proof is completely new and based on \cite{MR1436822}.
	Notice that by scaling invariance, we may assume $R=1$ without loss of generality. 
	Also, all constants in this section depend only on $n$ and $p$.
	
	\begin{lemma}\label{lm:upperestimatenonsingular}
		Let $u\in C^{6}(B_{1}) \cap$ $C(\bar{B}_{1})$ be a positive non-singular solution to \eqref{ourPDE} with $p\in(1,2^{\#}-1)$ and $R=1$. 
		Then,
		there exists $C>0$ such that
		\begin{equation}\label{upperestimatenonsingular}
			u(x)\leqslant C(1-|x|)^{-\gamma_p} \quad {\rm in} \quad B_{1}.
		\end{equation}
	\end{lemma}
	
	\begin{proof}
		See \cite[Lemma~3.1]{arXiv:2210.04619}.
	\end{proof}
	
	Using the last lemma, we can prove the auxiliary result below
	\begin{lemma}\label{lm:universalestimates}
		Let $u\in C^{6}(B_1^*)$ be a positive singular solution to \eqref{ourPDE} with $p\in(1,2^{\#}-1)$ and $R=1$. 
		Then, there exists $C>0$ such that
		\begin{equation*}
			|u(x)| \leqslant C|x|^{-\gamma_p} \quad {\rm in} \quad B^*_{1/2}.
		\end{equation*}
	\end{lemma}
	
	\begin{proof}
		Fixing $x_{0} \in B^*_{1/2}$, let us define $r=\frac{1}{2}|x_{0}|$.
		Then, since $\bar{B}_{r}\left(x_{0}\right) \subset B^*_{1}$, it is well-defined the rescaled function given by 
		\begin{equation*}
			\widetilde{u}_r(x)=r^{\gamma_p}u\left(rx+x_{0}\right) \quad \mbox{in} \quad \bar{B}_{1}.
		\end{equation*}
		Since $u$ is a positive singular solution to \eqref{ourPDE}, we obtain that $\widetilde{u}_r$ is a positive non-singular solution to
		\begin{equation*}
			(-\Delta)^3\widetilde{u}_r=f_p(\widetilde{u}_r) \quad \mbox{in} \quad B_1,
		\end{equation*}
		which is continuous up to the boundary.
		Therefore, we can apply Lemma~\ref{lm:upperestimatenonsingular} to $\widetilde{u}_r$, which, by taking $x=0$ in \eqref{upperestimatenonsingular}, provides $|\widetilde{u}_r(0)|\leqslant C$. 
		Hence, by rewriting in terms of $u$, we get $|{u}(x_{0})| \leqslant C r^{-\gamma_p}$.
		At last, since $x_{0} \in B^*_{1/2}$ is arbitrary and $r=\frac{1}{2}|x_{0}|$, the proof is finished.
	\end{proof}
	
	As a consequence, we have the following Harnack-type estimate
	\begin{corollary}\label{cor:universalestimates}
		Let $u\in C^{6}(B_1^*)$ be a positive singular solution to \eqref{ourPDE} with $p\in(1,2^{\#}-1)$ and $R=1$.
		Then,
		there exists $C>0$ (depending only on $n$ and $p$) such that
		\begin{equation*}
			\sum_{j=0}^{5}|x|^{\gamma_p+j}\left|D^{(j)} u(x)\right| \leqslant C \quad {\rm in} \quad B^*_{1/2}.
		\end{equation*}
	\end{corollary}
	
	\begin{proof}
		See \cite[Theorem~1.1]{arXiv:2210.04619}.
	\end{proof}
	
	The last lemma is a Harnack-type inequality
	\begin{lemma}\label{lm:gradientestimates}
		Let $u\in C^{6}(B_1^*)$ be a positive singular solution to \eqref{ourPDE} with $p\in(1,2^{\#}-1)$ and $R=1$.
		Assume that $-\Delta u\geqslant0$ and $\Delta^2u\geqslant0$.
		Then, there exist $c_1,c_2>0$ and $r_0\in(0,1/16)$ such that
		\begin{equation*}
			\sup_{B_{r/2}\setminus \bar{B}_{3r / 2}} u\leqslant c_1\inf _{B_{r/2}\setminus\bar{B}_{3r/2}} u \quad {\rm in} \quad B^*_{1/16}.
		\end{equation*}
		Moreover, the constant $c_1>0$ depends only on $n$, and $p$.
	\end{lemma}
	
	\begin{proof}
		See \cite[Proposition~4.2]{arXiv:2210.04619}.
	\end{proof}
	
	We also establish some auxiliary results about upper bounds near the isolated singularity. 
	\begin{lemma}\label{lm:uniformestimategidasspruck}
		Let $u\in C^{6}(B_1^*)$ be a positive singular solution to \eqref{ourPDE} with $p\in(2_{\#},2^{\#}-1)$ and $R=1$, and $v=\mathfrak{F}^{-1}(u)$ be its autonomous Emden--Fowler transformation given by \eqref{cyltransform}.
		Then, there exists $C>0$ such that
		\begin{equation*}
			|v|+|v^{(1)}|+|v^{(2)}|+|v^{(3)}|+|v^{(4)}|+|v^{(5)}|+|\nabla_{\theta}v|+|\Delta_{\theta}v|+|\nabla_{\theta}\Delta_{\theta}v|+|\Delta^2_{\theta}v|\leqslant C \quad {\rm in} \quad \mathcal{C}_{\ln2}.	\end{equation*}
	\end{lemma} 
	
	\begin{proof}
		First, by Lemma~\ref{lm:universalestimates}, we know that $v \in C^{6}(\mathcal{C}_T)$ is uniformly bounded. 
		Moreover, using Corollary~\ref{cor:universalestimates}, one can find $C>0$ such that
		\begin{align*}
			|v^{(1)}|+|\nabla_{\theta} v|& \leqslant C \sum_{j=0}^{1}|x| ^{\gamma_p}|D^{(j)} u(x)| \leqslant C, \\
			|v^{(2)}|+|\Delta_{\theta} v| & \leqslant C \sum_{j=0}^{2}|x|^{\gamma_p+j}|D^{(j)}u(x)| \leqslant C, \\
			|v^{(3)}|+|\nabla_{\theta}\Delta_{\theta} v|& \leqslant C \sum_{j=0}^{4}|x|^{\gamma_p+j}|D^{(j)} u(x)| \leqslant C,\\
			|v^{(4)}|+|\Delta^2_{\theta} v| & \leqslant C \sum_{j=0}^{4}|x|^{\gamma_p+j}|D^{(j)}u(x)| \leqslant C, \\
			|v^{(5)}|& \leqslant C \sum_{j=0}^{5}|x|^{\gamma_p+j}|D^{(j)} u(x)| \leqslant C,\\
		\end{align*}
		for $0<|x|<1/2$, which by a direct rescaling argument proves the lemma.
	\end{proof}
	
	In the lower critical case, we have the asymptotic upper bound below.
	Our approach to proving this result follows the same lines of \cite[Theorem~5]{MR1436822}.
	Nonetheless, we write a short proof for the convenience of the reader.
	\begin{lemma}\label{lm:sharpestimate}
		Let $u\in C^{6}(\mathbb R^n\setminus\{0\})$ be a positive  singular solution to \eqref{ourPDE} with $p=2_{\#}$.
		Assume that $-\Delta u\geqslant0$ and $\Delta^2u\geqslant0$.
		Then, there exist $C_0(n)>0$ and $0<r_0<R$ such that
		\begin{equation*}
			|\overline{u}(x)|\leqslant C_0(n)|x|^{6-n}(\ln|x|)^{\frac{6-n}{6}} \quad {\rm for} \quad 0<|x|<r_0.
		\end{equation*}
	\end{lemma}
	
	\begin{proof}
		Note that for each $0<r<R$ with $r=|x|$, we get that the spherical average $\overline{u}\in C^{6}(\mathbb R)$ satisfies the following nonautonomous ODE
		\begin{equation}\label{polarODE}
			r^{-6}\partial_r^{(6)}+M_5(n,r)\partial_r^{(5)}+M_4(n,r)\partial_r^{(4)}+ M_3(n,r)\partial_r^{(3)}+M_2(n,r)\partial_r^{(2)}+M_1(n,r)\partial_r-f_{\#}(u)\geqslant0,
		\end{equation}
		where the coefficients $M_j(n,\cdot):(0,R)\rightarrow\mathbb{R}$ are given by \eqref{laplaciancoefficients1} and \eqref{laplaciancoefficients2}.
		Next, setting 
		\begin{equation*}
			\psi_0=u, \quad \psi_1=-\Delta u, \quad {\rm and} \quad \psi_2=\Delta^2 u.
		\end{equation*}
		we can reformulate \eqref{polarODE} as a system in radial coordinates and the divergent form, we have
		\begin{equation}\label{radialpolarsystem}
			\begin{cases}
				-\left(r^{n-1} \psi_0^{(1)}(r)\right)^{(1)}=r^{n-1} \psi_0(r) &\\
				-\left(r^{n-1} \psi_1^{(1)}(r)\right)^{(1)}=r^{n-1} \psi_1(r) & \text { for } r \in(0, R)\\
				-\left(r^{n-1} \psi_2^{(1)}(r)\right)^{(1)}=r^{n-1} f_{\#}(u(r)). &  
			\end{cases}
		\end{equation}
		
		In what follows, the proof will be divided into some steps
		
		\noindent{\bf Step 1:}  Either $u\in C(B_1)\cap H^3(B_1)$ is a continuous weak solution to \eqref{ourPDE} with $p=2_{\#}$, or $\lim_{r\rightarrow0}u(r)=\lim_{r\rightarrow0}-\Delta u(r)\lim_{r\rightarrow0}\Delta^2 u(r)=+\infty$. 
		In particular, 
		\begin{equation*}
			u^{(1)}(r) \leqslant 0, \quad (-\Delta u)^{(1)}(r) \leqslant 0, \quad {\rm and} \quad (\Delta^2 u)^{(1)}(r) \leqslant 0 \quad {\rm for} \quad r\in(0, R).
		\end{equation*}
		
		Suppose that $u\in C^1(B^{*}_1)$ does not have a removable singularity at the origin. 
		Then $u$ must be unbounded in $B_1$, and, in this case, also a distribution solution to \eqref{ourPDE} in the entire $B_1$. 
		Since $f_p(u) \in L^{n/(6-\delta)}(B_1)$ for some $\delta \in(0,6)$, by bootstrap argument, $u$ may be extended to a continuous weak solution to \eqref{ourPDE}, which is a contradiction.
		
		First, we claim that $u(r)\rightarrow+\infty$ as $r \rightarrow 0$. 
		In fact, suppose by contradiction that $\liminf_{r \rightarrow 0} u(r)<+\infty$.
		Then there exist two real sequences $\{m_k\}_{k\in\mathbb N},\{m_k\}_{k\in\mathbb N}$ such that $M_k>m_k\rightarrow 0$ as $k\rightarrow\infty$ and $u$ assumes local maxima at $M_k$, local minima at $m_k$ for all $k\in\mathbb N$.
		Moreover $u(M_k) \rightarrow +\infty$ as $k\rightarrow\infty$ while $\{u(m_k)\}_{k\in\mathbb N}\subset \mathbb R$ remains bounded. 
		Since $-\Delta u \geqslant 0$, by the maximum principle, we get 
		\begin{equation*}
			u(M_k) \geqslant \min \{u(m_k), u(m_{k+1})\}\quad {\rm for \ all} \quad k\in\mathbb N,
		\end{equation*}
		which contradicts the boundedness of $\{u(m_k)\}_{k\in\mathbb N}\subset \mathbb R$.
		Using the same argument, it is also clear that $-\Delta u(r)\rightarrow+\infty$ and $\Delta^2 u(r)\rightarrow+\infty$  as $r \rightarrow 0$. 
		
		By integrating \eqref{radialpolarsystem} on $(\rho, r)$, where $0<\rho\ll R$ small enough, we find 
		\begin{equation*}
			r^{n-1} \psi_j^{(1)}(r) \leqslant \rho^{n-1} \psi_j^{(1)}(\rho)<0 \quad {\rm for} \quad j=0,1,2.
		\end{equation*}
		This completes the proof of Step 1.

		To continue with the proof, we need to introduce some auxiliary functions.
		Namely, we define the functions $\Psi_0,\Psi_1,\Psi_2:(0,R)\rightarrow \mathbb R$ given by
		\begin{equation*}
			\Psi_0(r)=r\psi_0(r), \quad \Psi_1(r)=r\psi_1(r)+(n-2)\psi^{(1)}_1, \quad \Psi_2(r)=r\psi_2+(n-2)\psi^{(1)}_2,
		\end{equation*}
		which satisfy
		\begin{equation}\label{radialsystemauxiliary}
			\begin{cases}
				-\Delta \Psi_0=\Psi_1 \\
				-r^{-1} \Psi_1^{(1)}=\Psi_2\\
				-r^{-1} \Psi_2^{(1)}=(-\Delta)^3\Psi_0.
			\end{cases}
		\end{equation}
		
		\noindent{\bf Step 2.} There exists $0<r_0\ll1$ and $0<\rho<r_0$, it follows
		$\lim_{r\rightarrow 0}r^{n-2(2-j)} \psi_j(r)=0$ and $\Psi_j(r)\geqslant 0$ in $r\in(0,\rho)$ for $j=0,1,2$. 
		
		Since $u\in L^1_{\rm loc}(B_R)$ is a distribution solution to \eqref{ourPDE} with $p=2_{\#}$, it follows that $\liminf_{r \rightarrow 0}r^{n-2(2-j)} \psi_j(r)=0$ for $j=0,1,2$.
		
		We start with the case $j=2$. Indeed, let us denote by $\psi_0^*(r^{-1})=r^{n-6} \psi_0(r)$ the sixth order Kelvin transform $\psi_0=u$ with respect to $\partial B_1$.
		Hence, by \cite[Lemma 3]{MR1436822}, we find that $-\Delta^3 \psi_0^* \geqslant 0$ in $\mathcal{I}^*(B_R):=\mathbb R^n\setminus B^*_1$.
		Then, by Step 1, $\psi_0^*$ is monotone near $+\infty$, which implies that $r^{n-6}\psi_0(r)$ is monotone near the origin. 
		Hence $\lim_{r \rightarrow 0} r^{n-6} \psi_0(r)>0$ exists and is positive for small $0<r\ll1$.
		
		To verify the case $j=1$, let us denote by $\psi_1^*(r^{-1})(r^{-1})=r^{n-4} u(r)$ the fourth order Kelvin transform of $\psi_1=\Delta u$ with respect to $\partial B_1$.
		Hence, by \cite[Lemma 3]{MR1436822}, we find that $\Delta^2 \psi_1^* \geqslant 0$ in $\mathcal{I}^*(B_R)$.
		Then, from Step 1, we conclude that $\psi_1^*$ is monotone near $+\infty$, which implies that $r^{n-4}\psi_1(r)$ is monotone near the origin. 
		Hence $\lim_{r \rightarrow 0} r^{n-4} \psi_1(r)>0$ exists and is positive for small $0<r\ll1$.
		
		Finally, for the case $j=0$, let us denote by $\psi_2^*(r^{-1})=r^{n-2} u(r)$ the second order Kelvin transform of $\psi_2=\Delta^2 u$ with respect to $\partial B_1$.
		Hence, by \cite[Lemma 3]{MR1436822}, we find that $-\Delta \psi_2^* \geqslant 0$ in $\mathcal{I}^*(B_R)$.
		Then, from Step 1, we conclude that $\psi_2^*$ is monotone near $+\infty$, which implies that $r^{n-2}\psi_2(r)$ is monotone near the origin. 
		Hence $\lim_{r \rightarrow 0} r^{n-2} \psi_2(r)>0$ exists and is positive for small $0<r\ll1$.
		
		The proof of Step 2 is concluded.
		
		\noindent{\bf Step 3.} $\lim_{r \rightarrow 0}r^{n-2} \Psi_0(r)=0$ and $\Psi_0^{(1)}(r)+(n-2)\Psi_0(r) \geqslant 0$.
		
		Finally,  from  the case $j=2$, it holds that $0 \leqslant-r u^{(1)}(r) \leqslant(n-6) u(r)$ in $(0,\rho)$.
		Thus, multiplying the last inequality by $r^{n-2}$ and letting $r \rightarrow 0$, we get that $\lim_{r \rightarrow 0} r^{n-1} u^{(1)}(r)=0$, and so
		\begin{equation*}
			\lim _{r \rightarrow 0} r^{n-2} \Psi_0(r)=\lim _{r \rightarrow 0}\left(r^{n-1} u^{(1)}(r)+(n-6) r^{n-2} u(r)\right)=0.
		\end{equation*}
		Next, for $0<r_0\ll1$ small given by Step 2, it follows from \eqref{radialsystemauxiliary} that $-\Delta\Psi_0\geqslant0$ and $\Psi_0\geqslant 0$ in $B^*_{r_0}$.
		Therefore, by considering $\Psi_0^*(r^{-1})=r^{n-2}\Psi_0(r)$, we obtain $-\Delta\Psi_0^* \geqslant 0$ and $\Psi_0^* \geqslant 0$ in $\mathcal{I}^*(B_R)$.
		In addition, $\Psi_0^*(r^{-1}) \rightarrow 0$.
		Consequently, a direct application of the maximum principle, yields $\Psi_0^*(r^{-1}) \leqslant 0$ for all $r^{-1}>r_0^{-1}$, which in turn proves completes the proof of Step 3.
		
		\noindent{\bf Step 4.} $\Psi_0^{(1)}(r) \leqslant 0$ for $0<r\ll1$, then $\Psi_0(r) \rightarrow+\infty$ as $r\rightarrow 0$.
		
		We observe that from Step 3 it follows that $-(r^{n-1} \Psi_0^{(1)})^{(1)} \geqslant 0$ for $r \in(0, \rho_1)$,
		which by integrating on $(\varepsilon, r)$, with $r\in(0,\rho_1)$, yields
		\begin{equation*}
			-r^{n-1} \Psi_0^{(1)}(r)+\varepsilon^{n-1} \Psi_0^{(1)}(\varepsilon)\geqslant 0.
		\end{equation*}
		Hence, by letting $\varepsilon \rightarrow 0$ in the last inequality we find that $\Psi_0^{(1)}(r) \leqslant 0$ for $r\in(0, \rho_1)$. 
		At last, suppose by contradiction that there exists $C>0$ such that $\Psi_0(r)\leqslant C$ for all $r>0$, which would imply that 
		\begin{equation*}
			\Psi_0(r)=r^{7-n}(r^{n-6} u(r))^{(1)} \leqslant C.
		\end{equation*}
		
		On the other hand, integrating the last inequality on $(0, r)$, we would obtain $u(r) \leqslant C(n-6)^{-1}$, which is a contradiction of the origin is a non-removable singularity and proves Step 4.
		
		\noindent{\bf Step 5.} $|\overline{u}(x)|\leqslant C_0(n)|x|^{6-n}(\ln|x|)^{\frac{6-n}{6}}$ for $0<|x|<r_0$.
		
		Let us set 
		\begin{equation*}
			\varphi_0(r)=r^{n-6} u(r), \quad \varphi_1(r)=r^{n-4} \varphi_0(r), \quad {\rm and} \quad \varphi_2(r)=r^{n-2} \varphi_1(r).
		\end{equation*}
		Whence, using Step 1, we get that $6\varphi_0^{(1)}(r) \geqslant 0$, $\varphi_1^{(1)}(r)$ and $\varphi_2^{(2)}(r) \geqslant 0$ in $(0, r_0)$. 
		Now, since $\Psi_1(r)=$ $r^{4-n}\varphi_1^{(1)}(r)$, one has from Step 2 the following holds
		\begin{equation*}
			-\left(r \Psi_0^{(1)}(r)+(n-2) \Psi_0(r)\right)^{(1)}=r^{4-n} \varphi_1^{(1)}(r),
		\end{equation*}
		which by integrating on $(r, r_0)$, yields
		\begin{equation*}
			-\left(r \Psi_0^{(1)}(r_0)+(n-2) \Psi_0(r_0)\right)+r \Psi_0^{(1)}(r)+(n-2) \Psi_0(r)\geqslant \varphi_1(r)\left(r^{4-n}-r_0^{4-n}\right).
		\end{equation*}
		Now, since $\varphi_0(r) \rightarrow 0$ as $r \rightarrow 0$, one can find $r_1 \in(0, r_0)$ such that
		\begin{equation*}
			\left(r \Psi_0^{(1)}(r_0)+(n-2) \Psi_0(r_0)\right)-\varphi_1(r) r_0^{4-n} \geqslant 0, \quad \text { on } \quad (0, r_1),
		\end{equation*}
		which, since from Step 5 implies $\Psi_0^{(1)}(r) \leqslant 0$ for $r$ small,  one can find $C>0$ such that 
		\begin{equation*}
			\Psi_0(r) \geqslant Cr^{4-n} \varphi_0(r) \quad {\rm in} \quad (0,r_1),
		\end{equation*}
		and so $\varphi_0^{(1)}(r) \geqslant Cr^{-1} \varphi_1(r)$.
		Hence, since $\psi_1(r) \geqslant Cf_{\#}(u(r))$, we get $\varphi_1(r) \geqslant Cf_{\#}(\varphi_0(r))$, which yields
		\begin{equation*}
			\varphi_0^{(1)}(r) \geqslant C r^{-1} f_{\#}(\varphi_0(r)) \quad {\rm for} \quad 0<r<r_1.
		\end{equation*}
		Up to a rescaling, we may assume that the last inequality holds for $r\in(0,1)$. 
		Thus, by a direct integration, we obtain $\varphi_0(r)^{-\frac{6}{n-6}} \geqslant C\ln r$ as $r\rightarrow0$,
		which in turn leads to
		\begin{equation*}
			u(r)r^{n-6}(\ln r)^{\frac{n-6}{6}} \leqslant C \quad {\rm as} \quad r\rightarrow0.
		\end{equation*}
		This completes the proof of the lemma.
	\end{proof}
	
	\section{Asymptotic radial symmetry}\label{sec:asympradialsymmetry}
	We prove the first part of Theorem~\ref{maintheorem} asserting that solutions to \eqref{ourPDE} are radially symmetric about the origin. 
	This symmetry will later be used to convert the singular PDE into a non-singular ODE on the cylinder.
	
	Before that, we need to establish some preliminaries.
	\subsection{Kelvin transform}
	Later we will employ the moving spheres technique, which is based on the {sixth order Kelvin transform} of a real valued function. 
	To define the Kelvin transform, we need to establish the concept of {inversion about a sphere} $\partial B_{\mu}(x_0)$, which is a map $\mathcal{I}_{x_0,\mu}:\mathbb{R}^n\rightarrow\mathbb{R}^n\setminus\{x_0\}$ given by $\mathcal{I}_{x_0,\mu}(x)=x_0+K_{x_0,\mu}(x)^2(x-x_0)$, where $K_{x_0,\mu}(x)=\mu/|x-x_0|$. 
	
	\begin{definition}
		For any $u\in C^6(B_R^*)$, let us consider the sixth order Kelvin transform about the sphere with center at $x_0\in\mathbb{R}^n$ and radius $\mu>0$ defined by 
		\begin{equation*}
			u_{x_0,\mu}(x)=K_{x_0,\mu}(x)^{n-6}u\left(\mathcal{I}_{x_0,\mu}(x)\right).
		\end{equation*}
	\end{definition}
	
	\begin{lemma}
		If $u\in C^6(B_R^*)$ is a solution to \eqref{ourPDE}, then  $u_{x_0,\mu}\in C^6(B_R^*\setminus\{x_0\})$ is a solution to 
		\begin{equation*}
			(-\Delta)^{3}u_{x_0,\mu}=K_{x_0,\mu}(x)^{(n-6)p-(n+6)}f_p(u_{x_0,\mu}) \quad {\rm in} \quad B_R^*\setminus\{x_0\}.
		\end{equation*}
	\end{lemma}
	
	\begin{proof}
		It directly follows by using \cite[Lemma~3.6]{MR1769247}.
	\end{proof}
	
	\subsection{Integral representation}
	Now we use a Green identity to transform the sixth order differential system  \eqref{ourPDE} into an integral system.
	In this way, we can avoid using the classical form of the maximum principle, and a sliding method is available \cite{MR3558255, MR2055032}, which will be used to classify solutions. 	
	Besides, in this setting is also possible to prove regularity through a barrier construction.
	
	We start with the following result 
	\begin{lemma}\label{lm:integrability}
		Let $u\in C^{6}(B_R^*)$ be a positive solution to \eqref{ourPDE} with $p\in(1,+\infty)$. 
		Assume that $-\Delta u\geqslant0$ and $\Delta^2u\geqslant0$.
		Then, $u \in L^{p}(B_{1})$. 
		In particular, if $p\in(2_{\#},+\infty)$, then $u \in L^1(B_1)$ is a
		distribution solution to \eqref{ourPDE}, that is, for all positive  $\phi\in C^{\infty}_c(B_1)$, one has
		\begin{equation*}
			\int_{B_{1}} u(-\Delta)^3 {\phi}\ud x=\int_{B_{1}}f_p(u) {\phi} \ud x.
		\end{equation*}
	\end{lemma}
	
	\begin{proof}
		See the proof of \cite[Lemma~3.1]{MR4123335} (see also \cite[Theorem~3.7]{MR1769247}). 
	\end{proof}
	The next result uses the Green identity to convert \eqref{ourPDE} into an integral system. We divide this results in two cases, namely $R<+\infty$ and $R=+\infty$.
	
	\begin{lemma}\label{lm:integralrepresentation}
		Let $u\in C^{6}(B_R^*)$ be a positive solution to \eqref{ourPDE} with $p\in(1,+\infty)$. 
		Assume that $-\Delta u\geqslant0$ and $\Delta^2u\geqslant0$.
		\begin{itemize}
			\item[{\rm (i)}] If $R<+\infty$, then $($up to constant$)$ there exists $r_0>0$ such that 
			\begin{equation}\label{integralsystem}
				u(x)=\int_{B_{r_0}}|x-y|^{6-n}{f}_p(u) \ud y+\psi(x),
			\end{equation}
			where $\psi>0$ satisfies $(-\Delta)^{3} \psi=0$ in $B_{r_0}$. Moreover, one can find a constant $C(\widetilde{r})>0$ such that $\|\nabla \ln \psi\|_{C^{0}(B_{\widetilde{r}})} \leqslant C(\widetilde{r})$ for all $0<\widetilde{r}<r_0$. 
			\item[{\rm (ii)}] If $R=+\infty$, then $($up to constant$)$, it follows
			\begin{equation*}
				u(x)=\int_{\mathbb{R}^n}|x-y|^{6-n}{f}_p(u) \ud y.
			\end{equation*}
			
		\end{itemize}		 
	\end{lemma}
	
	\begin{proof}
		For (i), the proof is a simple adaptation of \cite[Lemma~2.3]{arxiv:1901.01678}.
		For (ii) see  \cite[Theorem~4.3]{MR2200258} (see also \cite{MR2131045}).		
	\end{proof}
	
	\subsection{Sliding technique}
	Now we use the preliminary results to run an integral moving spheres technique.
	We use an asymptotic moving spheres technique in the same spirit of \cite{MR4085120}.
	Although our proof are almost the same we include it here for the sake fo completeness.
	
	\begin{proposition}\label{prop:asymptoticsymmetry}
		Let $u\in C^{6}(B_R^*)$ be a positive singular solution to \eqref{ourPDE} with $p\in(1,2^{\#}-1)$. 
		Assume that $-\Delta u\geqslant0$ and $\Delta^2u\geqslant0$.
		Then,
		\begin{equation*}
			u(x)=(1+\mathcal{O}(|x|))\overline{u}(x) \quad {\rm as} \quad x\rightarrow0.
		\end{equation*}
	\end{proposition}
	
	\begin{proof}
		Initially, if the origin is a removable singularity, then the conclusion is clear. Hence, we suppose that the origin is a non-removable singularity.
		
		We divide the proof into some claims.
		
		\noindent{\bf Claim 1:} There exists small $0<\varepsilon\ll1$ such that for any $z \in B^*_{\varepsilon/2}$,  it holds
		\begin{equation}\label{asymptoticmovingplanes}
			u_{z, r}\leqslant u  \quad {\rm in} \quad B_{1} \setminus\left(B_{r}(z)\cup\{0\}\right) \quad \mbox{for} \quad 0<r \leqslant|z|.
		\end{equation}
		
		\noindent Indeed, the proof follows almost the same lines as the one in \cite[Lemma~3.2]{arxiv:1901.01678}, so we omit it.
		
		In the next claim, we provide some estimates to be used later in the proof.
		
		\noindent{\bf Claim 2:} There exists $z\in B^*_{\varepsilon/2}$, $0<r<|z|$ and $\mu_*\gg1$ large such that
		\begin{equation}\label{estimatesradius}
			\frac{y_{\mu}}{|y_{\mu}|^{2}}-z=\left(\frac{r}{\left|y /|y|^{2}-z\right|}\right)^{2}\left(\frac{y}{|y|^{2}}-z\right) \quad {\rm and} \quad \frac{|y_{\mu}|}{|y|} \leqslant \frac{1}{r}\left|\frac{y}{|y|^{2}}-z\right| \quad {\rm for \ any} \quad \mu>\mu_*.
		\end{equation}
		Here $y_{\mu}=y+2(\mu-y\cdot{\bf e}){{\bf e}}$ is the reflection of $y$ about the hyperplane $\partial H_{\mu}({\bf e})$, where $H_{\mu}({\bf e})=\{x: \langle x,{\bf e}\rangle>\mu\}$ and ${\bf e}\in\mathbb{S}^{n-1}$. In other words, ${y_{\mu}}{|y_{\mu}|^{-2}}$ is the reflection point of ${y}{|y|^{-2}}$ about $\partial B_{r}(z)$. 
		
		\noindent As matter of fact,  choosing $r=|z|$, it involves an elementary computation, as follows
		\begin{equation*}
			z=\frac{y}{|y|^{2}} +\frac{\left|y_{\mu}\right|^{2}}{|y|^{2}-\left|y_{\mu}\right|^{2}}\left(\frac{y}{|y|^{2}} -\frac{y_{\mu}}{\left|y_{\mu}\right|^{2}} \right)=\frac{\left(y-y_{\mu}\right)}{|y|^{2}-\left|y_{\mu}\right|^{2}}.
		\end{equation*}
		
		Next, we establish a comparison involving the Kelvin transform of a component solution with itself.
		
		\noindent{\bf Claim 3:} For any $\mu>\frac{1}{\varepsilon}$ and ${\bf e}\in\partial B_{1}$, if $\langle x,{\bf e}\rangle>\mu$ and $|y_{\mu}|>1$, it holds $u_{0,1}(y) \leqslant u_{0,1}\left(y_{\mu}\right).$
		
		\noindent In fact, to prove the last inequality, let us note first that $y\in B_{1/\varepsilon}$, if and only if, ${y}{|y|^{-2}}\in B_{\varepsilon}$.
		Now given $y\in \mathbb{R}^{n}$ such that $\langle y,{\bf e}\rangle>\mu$, $|y_{\mu}|>1$ and $0<r<|z|<{\varepsilon}/{2}$ satisfying \eqref{estimatesradius}. Let us define $x={y}{|y|^{-2}}$ and $x_{z,r}={y_{\mu}}{|y_{\mu}|^{-2}}$.
		Then, since $\langle y,{\bf e}\rangle>\mu>{\varepsilon}^{-1}$ and $\|y_{\mu}|>1$, we have $x \in B_{r}(z)$ and $x_{z,r}\in B_{1}\setminus B_{r}(z)$. Hence, using \eqref{asymptoticmovingplanes} and \eqref{estimatesradius}, we find  
		\begin{equation*}
			u_{0,1}(y)\leqslant u_{0,1}(y_{\mu}),
		\end{equation*}
		which proves the claim.
		
		Ultimately, using Claim 3, we invoke \cite[Theorem~6.1 and Corollary~6.2]{MR982351} to find
		$C>0$, independent of $\varepsilon>0$, such that if $|y|\geqslant |x|+C\varepsilon^{-1}$, it follows $u_{0,1}(y) \leqslant u_{0,1}(x)$.
		Therefore, since $u_{0,1}$ is positive and satisfies $-\Delta u_{0,1}\geqslant0$ and $\Delta^2u_{0,1}\geqslant0$, the last inequality implies
		\begin{equation*}
			u_{0,1}(|x|)=\left(1+\mathcal{O}\left(\frac{1}{R}\right)\right)\left(\inf _{\partial B_{R}} u_{0,1}\right) \quad {\rm as} \quad R \rightarrow +\infty,
		\end{equation*} 
		uniformly on $\partial B_{R}$, which in terms of $u$ implies the desired asymptotic radial symmetry with respect to the origin, and
		the proof is concluded.
	\end{proof}
	
	\section{The limiting Pohozaev levels}\label{sec:limitinglevels}
	After proving the radial symmetry of singular solutions to \eqref{ourPDE}, we shall classify them in the blow-up and shrink-down limit.
	The idea is to use a blow-up/shrink-down analysis, which comes from tangent cone techniques from minimal hypersurface theory, and will be described in the sequel.
	
	For any $u$ solution to \eqref{ourPDE} and $\lambda>0$, let us define the following $\lambda$-rescaling solution given by
	\begin{equation}\label{scalingfamily}
		\widehat{u}_{\lambda}(x):=\lambda^{\gamma_p}u(\lambda x),
	\end{equation}
	where we recall that $\gamma_p={6}{(p-1)^{-1}}$.
	Notice that the $\lambda$-rescaled solution is still a positive  solution to \eqref{ourPDE} with $R=\lambda^{-1}$.
	Moreover, we get the following scaling invariance
	\begin{equation}\label{scaleinvariance}
		\mathcal{P}_{\rm sph}(r,\widehat{u}_{\lambda})=\mathcal{P}_{\rm sph}(\lambda r,u).
	\end{equation}
	This follows by directly using the inverse cylindrical transform as in Remark~\ref{rmk:limitrelation}.
	Besides, by a blow-up (resp. shrink-down) solution $u_0$ (resp. $u_{\infty}$) to \eqref{ourPDE}, we mean the limit $u_{0}:=\lim_{\lambda\rightarrow0}\widehat{u}_{\lambda}$ (resp. $u_{\infty}:=\lim_{\lambda\rightarrow+\infty}\widehat{u}_{\lambda}$).
	In fact, utilizing some a priori estimates and the compactness of the family $\{\widehat{u}_{\lambda}\}_{\lambda>0}\subset C^{6,\alpha}_{\rm loc}(\mathbb{R}^n)$ for some $\alpha\in(0,1)$, these limits will be proven to exist.
	Next, we study the limit Pohozaev functional both as $r\rightarrow0$ (blow-up) and $r\rightarrow+\infty$ (shrink-down), this will give the desired information about the asymptotic behavior for solutions to \eqref{ourPDE}.
	
	Here is our main result of this section:
	\begin{proposition}\label{prop:limitlocalbehavior}
		Let $u$ be a positive  singular solution to \eqref{ourlimitPDE} with $p\in(1,2^{\#}-1)$.
		Assume that $u$ is homogeneous of degree $-\gamma_p$.
		\begin{itemize}
			\item[{\rm (a)}] If $p\in(1,2_{\#}]$, then $u\equiv0$; 
			\item[{\rm (b)}] If $p\in(2_{\#},2^{\#}-1)$, then $\mathcal{P}_{\rm cyl}(r,u)$ converges both as $r\rightarrow0$ and $r\rightarrow+\infty$, namely
			\begin{equation}\label{limintingconstant}
				\{\mathcal{P}_{\rm sph}(0, u),\mathcal{P}_{\rm sph}(+\infty, u)\}=\{-\omega_{n-1}\ell^*_p,0\}, \quad \mbox{where} \quad \ell^*_p=\frac{p-1}{2(p+1)}K_0(n,p)^{\frac{p+1}{p-1}}.
			\end{equation} 
		\end{itemize}    
	\end{proposition}
	
	First, we prove that the invariance of the Pohozaev invariant is equivalent to the homogeneity of the blow-up limit solutions to \eqref{ourlimitPDE}.
	A similar result can also be found in \cite{MR3190428}, where a different type of functional is considered.
	
	\begin{lemma}\label{lm:homogeneity}
		Let $u$ be a positive singular solution to \eqref{ourlimitPDE} with $p\in(2_{\#},2^{\#}-1)$.
		Then, $\mathcal{P}_{\rm sph}(r,\mathcal{U})$ is constant, if and only if, for any $r\in(r_{1},r_{2})$  with $0<r_1\leqslant r_2<R<+\infty$, $u$ is homogeneous of degree $-\gamma_p>0$ in $B_{r_{2}}\setminus\bar{B}_{r_{1}}$, that is,
		\begin{equation*}
			u(x)=|x|^{-\gamma_p} u\left(\frac{x}{|x|}\right) \quad {\rm in} \quad B_{r_{2}}\setminus\bar{B}_{r_{1}}.
		\end{equation*}
	\end{lemma}
	
	\begin{proof}
		Notice that if $p\neq 2^{\#}-1$, then $K_0(n,p)\neq 0$.
		Thus, supposing that $\mathcal{P}_{\rm sph}(r,u)$ is constant for $r_{1}<r<r_{2}$, together with Remark~\ref{rmk:pohozaevspherical} yields that $\partial_{\nu}u=\gamma_pr^{-1}u$ on $\partial B_{r}$ for any $r_{1}<r<r_{2}$, where $\nu$ is the unit normal pointing towards the origin. 
		Therefore, $u$ is homogeneous of degree $-\gamma_p$ in $B_{r_{2}}\setminus\bar{B}_{r_{1}}$, which concludes the proof.
	\end{proof}
	
	The following lemma provides an upper bound estimate for singular solutions to \eqref{ourPDE}.
	
	\begin{lemma}\label{lm:limituppperestiaamte}
		Let $u$ be a positive  singular solution to \eqref{ourlimitPDE} with $p\in(1,2^{\#}-1)$.
		Then,
		\begin{equation*}
			u(x) \leqslant \left(\frac{p-1}{2n}\right)^{-\frac{1}{p-1}}|x|^{-\gamma_p} \quad {\rm in}  \quad \mathbb{R}^{n} \setminus\{0\}.
		\end{equation*}
	\end{lemma}
	
	\begin{proof}
		By \cite[Theorem~6.1]{MR4438901}, we know that $-\Delta u\geqslant 0$ and $\Delta^2 u\geqslant 0$ in $\mathbb{R}^{n} \setminus\{0\}$, which, by using the extended maximum principle \cite[Theorem~1]{MR1814364}, gives us
		\begin{equation*}\label{limitlevel1}
			\liminf_{x \rightarrow 0} u(x)>0.
		\end{equation*}
		Considering $\varphi=u^{1-p}$, a direct computation, provides
		\begin{equation*}
			\Delta \varphi \geqslant \frac{p}{p-1} \frac{|\nabla \varphi|^{2}}{\varphi}+p-1 \quad \mbox{in} \quad \mathbb{R}^{n} \setminus\{0\}.
		\end{equation*}
		Thus, for any $r>0$, let us consider 
		the auxiliary function $\widetilde{\varphi}(x)=\varphi(x)-\frac{p-1}{2n}|x|^{2}$, which satisfies $-\Delta \widetilde{\varphi}\leqslant0$ in $B^*_{r}$.
		Furthermore, \eqref{limitlevel1} implies that $\widetilde{\varphi}$ is bounded close to the origin, and thus, again, by the extended maximum principle, we find
		\begin{equation*}
			0 \leqslant\limsup_{x \rightarrow 0} \widetilde{\varphi}(x) \leqslant \sup_{\partial B_{r}} \widetilde{\varphi}=\sup_{\partial B_{r}} \varphi-\frac{p-1}{2n} r^{2},
		\end{equation*}
		which yields
		\begin{equation*}
			\inf_{\partial B_{r}} u \leqslant\left(\frac{p-1}{2 n}\right)^{-\frac{1}{p-1}} r^{-\gamma_p}.
		\end{equation*}
		Finally, a direct application of Proposition~\ref{prop:asymptoticsymmetry} finishes this proof.
	\end{proof}
	
	As a consequence of this uniform upper bound, we prove the compactness of the family $\{\widehat{u}_{\lambda}\}_{\lambda>0}\subset C^{6,\alpha}_{\rm loc}(\mathbb{R}^n)$, for some $\alpha\in(0,1)$, which provides the existence of both blow-up and shrink-down limits for the scaling family defined by \eqref{scalingfamily}.
	
	\begin{lemma}\label{lm:limitcompactness}
		Let $u$ be a positive singular solution to \eqref{ourlimitPDE} with $p\in(1,2^{\#}-1)$. 
		Then, $\{\widehat{u}_{\lambda}\}_{\lambda>0}\subset C^{6,\alpha}_{\rm loc}(\mathbb{R}^n)$ is uniformly bounded, for some $\alpha\in(0,1)$.
	\end{lemma}
	
	\begin{proof}
		If the origin is a removable singularity of $\widehat{u}_{\lambda}$ for all $\lambda>0$, according to Theorem~\ref{thm:luo-wei}, $u$ is trivial and the conclusion follows.
		
		On the other hand, assuming that the origin is a removable singularity, Lemma~\ref{lm:limituppperestiaamte} provides that $\{\widehat{u}_{\lambda}\}_{\lambda>0}$ is globally bounded in $\mathbb{R}^{n}\setminus\{0\}$. Thus, we know that $\{\widehat{u}_{\lambda}\}_{\lambda>0}$ is uniformly bounded in each compact subset of $K\subset\mathbb{R}^{n} \setminus\{0\}$.
		Moreover, since for each $\lambda>0$, the scaling $\widehat{u}_{\lambda}$ also satisfies \eqref{ourlimitPDE}, it follows from standard elliptic estimates that $\{\widehat{u}_{\lambda}\}_{\lambda>0}$ is uniformly bounded in $C^{6,\alpha}(K)$, for some $\alpha\in(0,1)$, which concludes the proof.
	\end{proof}
	
	Recall that $\mathcal{P}_{\rm sph}(r, u)$ is the Pohozaev functional introduced in Remark~\ref{rmk:pohozaevspherical}, which by Proposition~\ref{lm:monotonicityformula} is monotonically nonincreasing in $r>0$ when $p\in(1,2^{\#}-1)$.
	
	\begin{lemma}\label{lm:limitlimitinglevels}
		Let $u_{0}$ $($or $u_{\infty}$$)$ be a positive singular blow-up $($or shrink-down$)$ solution to \eqref{ourPDE} under the family of scalings $\{\widehat{u}_{\lambda}\}_{\lambda>0}$. 
		Then,
		$\mathcal{P}_{\rm sph}(r, u_{0})\equiv\mathcal{P}_{\rm sph}(0, u)$ $($or $\mathcal{P}_{\rm sph}(r, u_{\infty})\equiv\mathcal{P}_{\rm sph}(\infty, u)$$)$ is constant
		for all $r>0$. In particular, both $u_{0}$ and $u_{\infty}$ are homogeneous of degree
		$-\gamma_p$.
	\end{lemma}
	
	\begin{proof}
		Let $\{\lambda_{k}\}_{k\in\mathbb{N}}\subset(0,+\infty)$  be a blow-up sequence such that $\lambda_{k}\rightarrow 0$, and  $u_{0}\in{C^{6,\alpha}(\mathbb{R}^n\setminus\{0\})}$ be its blow-up limit, that is, $\widehat{u}_{\lambda_k}\rightarrow u_{0}$ in ${C_{\rm loc}^{6,\alpha}(\mathbb{R}^n\setminus\{0\})}$ as $k\rightarrow+\infty$, for some $\alpha\in(0,1)$.
		Now, using Proposition~\ref{lm:monotonicityformula} and Lemma~\ref{lm:limitcompactness}, one concludes that there exists the limiting level $\mathcal{P}_{\rm sph}(0, u):=\lim_{r \rightarrow 0} \mathcal{P}_{\rm sph}(r, u)$.
		Moreover, due to the scaling invariance of the Pohozaev functional in \eqref{scaleinvariance}, for any $r>0$, it follows
		\begin{equation*}
			\mathcal{P}_{\rm sph}(r, u_{0})=\lim_{k \rightarrow +\infty}\mathcal{P}_{\rm sph}(r, \widehat{u}_{\lambda_k})=\lim_{k \rightarrow +\infty} \mathcal{P}_{\rm sph}(r\lambda_{k}, u)=\mathcal{P}_{\rm sph}(0,u),
		\end{equation*}
		which finishes the proof of the first assertion.
		Now, we can check that the homogeneity follows from Lemma~\ref{lm:homogeneity}.
		Finally, notice that the same argument can readily be employed, replacing the blow-up limit by the shrink-down limit, so we omit it here.
	\end{proof}
	
	\begin{lemma}\label{lm:blowupclassification}
		Let $u$ be a positive singular solution to \eqref{ourlimitPDE} with $p\in(1,2^{\#}-1)$. 
		Assume that $u$ is homogeneous of degree $-\gamma_p$.
		\begin{itemize}
			\item[{\rm (a)}] If $p\in(1,2_{\#}]$, then $u\equiv0$.
			\item[{\rm (b)}] If $p\in(2_{\#},2^{\#}-1)$, then either $u\equiv0$, or $u\equiv K_0(n,p)^{\frac{1}{p-1}}|x|^{-\gamma_p}$.
		\end{itemize}
	\end{lemma} 
	
	\begin{proof}
		Since $u$ is homogeneous of degree $-\gamma_p$, the Emden--Fowler transformation $v=\mathfrak{F}(u)$ given by \eqref{cyltransform} satisfies
		\begin{equation}\label{angularsystem}
			-\mathscr{L}_{\theta}v+f_p(v)=0 \quad \mbox{on} \quad \mathbb{S}_t^{n-1},
		\end{equation}
		where $\mathscr{L}_{\theta}:=\Delta_{\theta}^3+L_0(n,p)\Delta_{\theta}^{2}+J_{0}(n,p)\Delta_{\theta}v+K_0(n,p)$.
		Now we divide the the proof into two cases:
		
		\noindent{\bf Case 1:} $p\in(1,2_{\#}]$.
		
		\noindent Initially, one can verify that 
		$\mathscr{L}_{\theta}(v)\leqslant 0$ on $\mathbb{S}_t^{n-1}$.
		Next, observe that $\mathscr{L}_{\theta}$ is the composition of three elliptic operators.
		This, together with \cite[Theorem~1.2]{arXiv:2210.04619}, implies that $v\in C^6(\mathcal{C}_T)$ does not attain any strict local minimum on $\mathbb{S}_t^{n-1}$. 
		Therefore, since $\mathbb{S}_t^{n-1}$ is a compact manifold, it follows that $v$ is constant, which yields $v\equiv v_0$.
		Nevertheless, using that $K_0(n,p) \leqslant 0$, any positive constant solution to \eqref{angularsystem} is trivial.
		By using the inverse of the Emden--Fowler transformation, it holds that $u$ is trivial on $\partial B_{1}$, which by the superharmonicity property, implies that $u$ is trivial in the whole domain.
		This conclusion finishes the proof of the first case, and so part (a) of the lemma follows.
		
		\noindent{\bf Case 2:} $p\in(2_{\#},2^{\#}-1)$.
		
		\noindent Assume that $u$ is a nontrivial limit solution in the punctured space. 
		Hence, since each component of $u$ is positive and satisfies $-\Delta u\geqslant0$ and $\Delta^2 u\geqslant0$, it quickly follows that $u>0$ in $\mathbb{R}^n\setminus\{0\}$. 
		By homogeneity, the origin is a non-removable singularity of $u\in C^6(\mathbb R^n\setminus\{0\})$. 
		Hence, by Proposition~\ref{prop:asymptoticsymmetry}, $u$ is radially symmetric; thus, $u\equiv u_0$ is a positive constant vector.
		Moreover, by \eqref{angularsystem}, it holds $u_0=K_0(n,p)^{\frac{1}{p-1}},$
		which, by using the homogeneity of $u$ and Lemma~\ref{lm:homogeneity}, finishes the proof of the second case, and so (b) holds.
	\end{proof}
	
	At last, we can prove the main result of this part.
	\begin{proof}[Proof of Proposition~\ref{prop:limitlocalbehavior}]
		Let $u_{0}$ and $u_{\infty}$ be, respectively, a blow-up and a shrink-down limit of $u$.
		According to Lemma~\ref{lm:limitinglevels} both $u_{0}$ and $u_{\infty}$ are homogeneous of degree $-\gamma_p$.
		In what follows, we divide the rest of the proof into two cases:
		
		\noindent{\bf Case 1:} $p\in(1,2_{\#}]$. 
		
		\noindent Here, it follows from Lemma~\ref{lm:blowupclassification} (a) that both $u_{0}$ and $u_{\infty}$ are trivial, which, by Lemma~\ref{lm:limitlimitinglevels}, provides $\mathcal{P}_{\rm sph}(0,u)=\mathcal{P}_{\rm sph}(+\infty, u)=0$. 
		In addition, using the monotonicity property of the Pohozaev functional, we find $\mathcal{P}_{\rm sph}(r, u)=0$ for all $r>0$. Hence, by Lemma \ref{lm:homogeneity}, $u$ is homogeneous of degree $-\gamma_p$.
		Therefore, the proof of (a) of Proposition~\ref{prop:limitlocalbehavior} is now an immediate consequence of Lemma~\ref{lm:blowupclassification} (a).
		
		\noindent{\bf Case 2:} $p\in(2_{\#},2^{\#}-1)$.
		
		\noindent Initially, by Lemmas~\ref{lm:limitlimitinglevels} and \ref{lm:blowupclassification} (b), any blow-up $u_{0}$ is either trivial or has the form \eqref{subcriticalblowupsolutions}.
		If $u_{0}$ is trivial, then clearly $\mathcal{P}_{\rm sph}(r, u_{0})=0$ for all $r>0$, which combined with Lemma~\ref{lm:limitlimitinglevels} implies that $\mathcal{P}_{\rm sph}(0, u)=0$. 
		Otherwise, a simple computation shows  $\mathcal{P}_{\rm sph}(r, u_{0})=-\omega_{n-1}\ell^*_p$ for all $r>0$.
		Therefore, using again Lemma~\ref{lm:limitlimitinglevels}, we have $\mathcal{P}_{\rm sph}(0, u)=-\ell^*_p$. 
		Since the converse trivially follows, we obtain that $\mathcal{P}_{\rm sph}(0, u) \in$ $\{-\omega_{n-1}\ell^*_p,0\}$.
		Moreover, $\mathcal{P}_{\rm sph}(0, u)=0$, if and only if, all the blow-ups are trivial, whereas $\mathcal{P}_{\rm sph}(0, u)=-\omega_{n-1}\ell^*_p$, if and only if, all the blow-ups are of the form \eqref{subcriticalblowupsolutions}.
		In the case of shrink-down $u_{\infty}$ solution, the strategy is similar, so we omit it.
		These conclusions finish the proof of Case 2, and therefore the proposition holds.
	\end{proof}
	
	\section{Local asymptotic behavior}\label{sec:localbehavior}
	In this section, we present the proof of Theorem~\ref{maintheorem}. 
	First, we show an asymptotic symmetry result, which permits us to migrate to an ODE setup. 
	Second, we prove some universal upper bound estimates, not depending on the superharmonic assumption.
	However, we should emphasize that in the rest of the argument, there is a significant change of behavior of radial solutions \eqref{ourPDEcyl} for distinct values of the power $p\in(1,2^{\#}-1]$. 
	This difference occurs due to the change of sign of the coefficients in the tri-Laplacian written in cylindrical coordinates.
	These signs control the Lyapunov stability of the solutions to linearized operator around a limit blow-up solution, and so the asymptotic behavior of the local solutions near the isolated singularity.
	
	We divide our argument into three subsections, where we prove, respectively, the local behavior near the isolated singularity for the situations: $p\in(1,2^{\#}-1)$ in Subsection~\ref{subsec:serrin-lions}, $p=2_{\#}$ in Subsection~\ref{subsec:gidas-spruck}, and $p\in(2_{\#},2^{\#}-1)$ in Subsection~\ref{subsec:aviles}. 
	
	\subsection{Serrin--Lions case}\label{subsec:serrin-lions}
	We prove Theorem~\ref{maintheorem} (a). 
	The asymptotic analysis for this case is straightforward.
	We are based in the approach given by \cite{MR4085120}.
	
	In the sequel, we aim to prove the following proposition
	\begin{proposition}\label{prop:serrinlionscase}
		Let $u\in C^{6}(B_R^*)$ be a positive singular solution to \eqref{ourPDE} with $p\in(1,2_{\#})$. 
		Assume that $-\Delta u\geqslant0$ and $\Delta^2u\geqslant0$. 
		Then, there exist $C_1,C_2>0$ $($depending on $u$$)$ such that $C_1|x|^{6-n}\leqslant u(x)\leqslant C_2|x|^{6-n}$ for $0<|x|\ll1$, or equivalently, $u(x)\simeq |x|^{6-n}$ as $x\rightarrow 0$.
	\end{proposition}
	
	First, we prove an upper bound estimate based on a Green identity from Lemma~\ref{lm:integralrepresentation} (i).
	
	\begin{lemma}\label{lm:upperestimateserrin}
		Let $u\in C^{6}(B_R^*)$ be a positive solution to \eqref{ourPDE} with $p\in(1,2_{\#})$. 
		Assume that $-\Delta u\geqslant0$ and $\Delta^2u\geqslant0$. 
		Then, there exists $C_2>0$, depending only on $u$, such that
		$u(x)\leqslant C_2|x|^{4-n}$ as $x\rightarrow 0$.
	\end{lemma}
	
	\begin{proof}
		Initially, by Lemma~\ref{lm:integrability}, we have that $u\in L^{p}(B_{1})$. 
		Moreover, since $p\in(1,2_{\#})$ and $u\in C^{6}(B_R^*)$ satisfies the Harnack inequality in Lemma~\ref{lm:gradientestimates}, it follows
		\begin{equation*}
			u(x)=\mathrm o\left(|x|^{-\gamma_p}\right) \quad {\rm as} \quad x\rightarrow 0,
		\end{equation*}
		which by $n-6<\gamma_p$, implies that for any $n-6<q<\gamma_p$, there exists $0<r_{q}<1$ depending only on $n$ $p$, and $q$ such that
		\begin{equation*}
			u(x)<|x|^{-q} \quad \mbox{in} \quad B^*_{r_{q}},
		\end{equation*}
		where in the last claim we have used a blow-up argument.
		Now taking $r_{q}>0$ as before, and using Lemma~\ref{lm:integrability} again, we get that $(-\Delta)^3 u=f_p(u) \in L^{1}(B_{1})$. 
		Thus, using \eqref{integralsystem}, we decompose
		\begin{equation}\label{lions2}
			u(x)=|x|^{6-n} -\int_{B_{r_q}}|x-y|^{6-n} (-\Delta)^3 u(y) \ud y+\psi(x) \quad \mbox{in} \quad B^*_{r_{q}},
		\end{equation}
		where $\psi\in C^{\infty}(B_{1})$ is such that $(-\Delta)^3\psi=0$ in $B_{r_{q}}$. 
		Nevertheless, using \eqref{integralsystem} that there exists $C_q>0$, depending only on $n$, $p$, and $q$ such that
		\begin{equation*}
			\left|\int_{B_{r_{q}}}|x-y|^{6-n}(-\Delta)^3 u(y) \ud y\right| \leqslant \int_{B_{r_{q}}}|x-y|^{6-n}|y|^{-pq} \ud y \leqslant C_{q}|x|^{6-n}.
		\end{equation*}
		Hence, fixing $n-6<q<\gamma_p$ and choosing suitable $r_{q}>0$ and $C_{q}>0$ on the last inequality, the proof follows directly from \eqref{lions2}.
	\end{proof}
	
	Second, we give a sufficient condition to classify whether the origin is a removable singularity or non-removable singularity.
	
	\begin{lemma}\label{lm:removabilityserrin}
		Let $u\in C^{6}(B_R^*)$ be a positive solution to \eqref{ourPDE} with $p\in(1,2_{\#})$. 
		Assume that $-\Delta u\geqslant0$ and $\Delta^2u\geqslant0$.
		If
		\begin{equation}\label{serrinestimate}
			u(x)=\mathrm o\left(|x|^{6-n}\right) \quad {\rm as} \quad x\rightarrow 0,
		\end{equation}
		then, the origin is a removable singularity.
	\end{lemma}
	
	\begin{proof}
		By \eqref{serrinestimate}, we get $u \in L^{q}(B_{1})$ for any $q\in[1,2_{\#})$.
		Moreover, since $p\in(1,2_{\#})$ and $|(-\Delta)^3u| \leqslant|u|^{p}$, it follows $(-\Delta)^3u \in L^{q/p}(B_{1})$ for any $q\in[1,2_{\#})$. 
		Whence, we can use standard elliptic theory 
		combined with a bootstrap argument to find $u\in W^{5,Q}(B_{1})$ for any $Q\in(1,+\infty)$. 
		In particular, it holds from Morrey's embedding that $u \in C^{4, \alpha}(B_{1})$ for any $\alpha\in(0,1)$. 
		Therefore, $u\in C^{6}(B_R)$, that is, it must have a removable singularity at the origin.
	\end{proof}
	
	Now we are in a position to prove our main result of this part.
	
	\begin{proof}[Proof of Proposition~\ref{prop:serrinlionscase}]
		Suppose by contradiction that $u\in C^{6}(B_R^*)$ has a non-removable singularity at the origin, that is, $u\in C^{6}(B_R)$.
		Then, using Lemma~\ref{lm:removabilityserrin}, we get that $u$ does not satisfy \eqref{serrinestimate}, that is, there exists $\rho>0$ and $\{r_{k}\}_{k\in\mathbb{N}}$ such that $r_{k}\rightarrow 0$ as $k\rightarrow+\infty$ satisfying 
		\begin{equation*}
			\sup_{\partial B_{r_{k}}}u \geqslant \rho r_{k}^{6-n}.
		\end{equation*}
		On the other hand, by the Harnack inequality in Lemma~\ref{lm:gradientestimates}, there exists $c_1>0$ satisfying $\inf_{\partial B_{r_{k}}} u\geqslant c_{1} \rho r_{k}^{6-n}$, where $c_{1}>0$ depends only on $n$, and $p$. 
		Taking $0<\rho\ll1$ smaller to ensure that there exists $c_2\rho \leqslant \inf_{\partial B_{1/2}}u$, it follows from the maximum principle that
		\begin{equation*}
			u(x) \geqslant c_{2} \rho|x|^{6-n} \quad \mbox{in} \quad B^*_{1/2},
		\end{equation*}
		which proves the asymptotic lower bound estimate in this case, and together with Lemma~\ref{lm:upperestimateserrin}, the proof of the proposition is concluded. 
	\end{proof}
	
	\subsection{Gidas--Spruck case}\label{subsec:gidas-spruck}
	The objective of this subsection is to prove Theorem~\ref{maintheorem} (c).
	Our strategy is based on the monotonicity formula for the Pohozaev functional in cylindrical coordinates (see Proposition~\ref{lm:monotonicityformula}), which relies on the strategy given in \cite{MR4123335,arXiv:2210.04619,10.1093/imrn/rnab212}.
	More precisely, we show that the local models near the origin are the limit blow-up solutions, whose limits are provided by its image under the action of the spherical Pohozaev functional.
	Finally, to prove the removability of the singularity theorem, we use a technique relying on the regularity lifting method from \cite{MR1338474}.
	
	We will prove the result below
	\begin{proposition}\label{prop:gidasspruckcase}
		Let $u\in C^{6}(B_R^*)$ be a positive singular solution to \eqref{ourPDE} with $p\in(2_{\#},2^{\#}-1)$. 
		Then, 
		\begin{equation*}
			u(x)=(1+\mathrm{o}(1))K_0(n,p)^{\frac{1}{p-1}}|x|^{-\gamma_p}.
		\end{equation*}
	\end{proposition}
	
	Now we use this rescaled family $\{\widehat{u}_{\lambda}\}_{\lambda>0}\subset C^{6,\alpha}(B^*_1)$, for some $\alpha\in(0,1)$, to obtain the blow-up limit for \eqref{ourPDE}. This allows us to study the limiting values for the Pohozaev functional, by using the classification results from Theorem~\ref{thm:luo-wei}.
	
	\begin{lemma}\label{lm:limitinglevesgidasspruck}
		Let $u\in C^{6}(B_R^*)$ be a positive singular solution to \eqref{ourPDE} with $p\in(2_{\#},2^{\#}-1)$ and $v={\mathfrak{F}}(u)$ be its autonomous Emden--Fowler transformation given by \eqref{newcylindrical}. 
		Then, ${\mathcal{P}}_{\rm cyl}(-\infty,v) \in\{-\ell^*_p,0\}$, where $\ell^*_p$ is given by \eqref{limintingconstant}.
		Moreover, it follows
		\begin{itemize}
			\item[{\rm (i)}] ${\mathcal{P}}_{\rm cyl}(-\infty,v)=0$, if and only if,
			\begin{equation}\label{asymptoticgidasspruck}
				u(x)=\mathrm o\left(|x|^{-\gamma_p}\right) \quad {\rm as} \quad x \rightarrow 0.
			\end{equation}
			\item[{\rm (ii)}] ${\mathcal{P}}_{\rm cyl}(-\infty,v)=-\ell^*_p$, if and only if,
			\begin{equation*}
				u(x)=(1+\mathrm{o}(1))K_0(n,p)^{\frac{1}{p-1}}|x|^{-\gamma_p} \quad {\rm as} \quad x \rightarrow 0.
			\end{equation*}
		\end{itemize}
	\end{lemma}
	
	\begin{proof}
		Initially, by Lemma~\ref{lm:universalestimates}, for any $K\subset B_{1/2\lambda}$ compact subset, the family 
		$\{\widehat{u}_{\lambda}\}_{\lambda>0}\subset C^{6,\alpha}(B^*_1)$ is uniformly bounded, for some $\alpha\in(0,1)$.
		Then, by standard elliptic theory, there exists a positive  function $u_{0} \in C^{6,\alpha}(\mathbb{R}^{n}\setminus\{0\})$, such that, up to a subsequence, we have that $\|\widehat{u}-u_{0}\|_{C_{\rm loc}^{6,\alpha}(\mathbb{R}^{n} \setminus\{0\})}$ as $\lambda \rightarrow 0$, where $u_{0}$ satisfies the blow-up limit system \eqref{ourPDE}.
		Moreover, by \cite[Theorem~1.2]{arXiv:2210.04619}, we know that $u_0\in C^6(\mathbb R^n\setminus\{0\})$ is satisfies $-\Delta u_0\geqslant 0$ and $\Delta^2 u_0\geqslant 0$ in $\mathbb{R}^{n} \setminus\{0\}$, which, by the maximum principle, yields that either $u_0 \equiv 0$
		or $u_0>0$
		in $\mathbb{R}^{n} \backslash\{0\}$.
		Therefore, by Theorem~\ref{thm:luo-wei}, the blow-up limit $u_{0}$ is radially symmetric with respect to the origin.
		Furthermore, by the scaling invariance of the Pohozaev functional, we get
		\begin{equation}\label{limitinggs}
			\mathcal{P}_{\rm sph}(r,u_{0})=\lim_{\lambda\rightarrow 0} \mathcal{P}_{\rm sph}(r, \widehat{u}_{\lambda})=\lim_{\lambda \rightarrow0} \mathcal{P}_{\rm sph}(\lambda r,u_0)=\mathcal{P}_{\rm sph}(0, u_0).
		\end{equation}
		
		In addition, since $v_0=\mathfrak{F}(u_0)$ satisfies \eqref{ourPDEcyl}, by \eqref{limitinggs}, we get that $\mathcal{P}_{\rm cyl}(t,v_{0})=\mathcal{P}_{\rm sph}(r,u_{0})$ is a constant.
		Consequently, by the monotonicity formula in Proposition~\ref{lm:monotonicityformula}, we get 
		\begin{equation*}
			\frac{\ud}{\ud t} \mathcal{P}_{\rm cyl}(t,v_{0})=\left[-K_{5}(n,p){v_0^{(3)}}^{2}+K_{3}(n,p){v_0^{(2)}}^{2}-K_{1}(n,p){v_0^{(1)}}^{2}\right] \equiv 0.
		\end{equation*}
		Moreover, since $K_5(n,p),K_{1}(n,p)>0$ and $K_{3}(n,p)<0$, we find that $v^{(1)}_0\equiv 0$ in $\mathbb{R}$, and so $v_0$ is constant, which can be directly computed, namely either $v_0=0$ or $v_0=K_{0}(n,p)^{\frac{1}{p-1}}$.
		Moreover, by \eqref{limitinggs}, it follows $\mathcal{P}_{\rm cyl}(0, v_0) \in\{-\ell^*_p ,0\}$ and $\mathcal{P}_{\rm sph}(0, u_0)\in\{-\omega_{n-1}\ell^*_p ,0\}$.
		
		Finally, if $\mathcal{P}_{\rm sph}(0, u_0)=0$, then, by uniqueness of the limit $u_0\equiv 0$. 
		Whence, we conclude that $\|\widehat{u}_{\lambda}\|_{C^{6,\alpha}(K)} \rightarrow 0$ for any sequence of $\lambda\rightarrow0$, for some $\alpha\in(0,1)$, which straightforwardly provides \eqref{asymptoticgidasspruck}.
		Otherwise, we have 
		\begin{equation*}
			u_0\equiv K_{0}(n,p)^{\frac{1}{p-1}}|x|^{-\gamma_p},
		\end{equation*}
		which proves (ii) of this lemma and finishes the proof.
	\end{proof}
	
	Next, we use the last lemma to prove the removable singularity theorem. 
	Our proof is based on regularity lifting methods combined with the De Giorgi--Nash--Moser iteration technique.
	
	\begin{lemma}\label{lm:removablesingularitygidasspruck}
		Let $u\in C^{6}(B_R^*)$ be a positive singular solution to \eqref{ourPDE} with $p\in(2_{\#},2^{\#}-1)$. 
		If
		\begin{equation*}
			u(x)=\mathrm o\left(|x|^{-\gamma_p}\right) \quad {\rm as} \quad x \rightarrow 0,
		\end{equation*}
		then the origin is a removable singularity.
	\end{lemma} 
	
	\begin{proof}
		Without loss of generality, let us consider $R=1$.
		The proof will be divided into some claims.
		
		\noindent{\bf Claim 1:} If
		$u(x)=\mathrm o(|x|^{-\gamma_p})$ as $x \rightarrow 0
		$, then
		\begin{equation}\label{condition}
			\int_{B_{1/2}}u^{n\gamma_p^{-1}}\ud x<+\infty.	
		\end{equation}
		
		\noindent In fact, let us consider $\phi(|x|)=|x|^{\zeta_p}$, where $\zeta_p=-2\gamma_p(2^{\#}-1)(n-2_{\#})(p-1)^{-1}$.
		Then, a direct computation, provides
		\begin{equation*}
			\Delta^{3}\phi(x)=\zeta_p(\zeta_p-2)(\zeta_p-3)(\zeta_p+n-2)(\zeta_p+n-4)(\zeta_p+n-6)|x|^{\zeta_p-6} \quad {\rm in} \quad \mathbb R^n\setminus\{0\},
		\end{equation*}
		which, since $\zeta_p+n-6=\gamma_p>0$, it follows that
		\begin{equation*}
			A_p:=\zeta_p(\zeta_p-2)(\zeta_p-3)(\zeta_p+n-2)(\zeta_p+n-4)(\zeta_p+n-6)>0.
		\end{equation*}
		Thus, we can write
		\begin{equation}\label{gs4}
			\frac{(-\Delta)^{3} \phi}{\phi}=\frac{A_p}{|x|^{6}} \quad \mbox{in} \quad \mathbb{R}^{n}\setminus\{0\}.
		\end{equation}
		
		For any $0<\varepsilon \ll 1$, let us consider $\eta_{\varepsilon}\in C^{\infty}\left(\mathbb{R}^{n}\right)$ with $0\leqslant\eta_{\varepsilon}\leqslant1$ a cut-off function satisfying
		\begin{equation}\label{cutoff2}
			\eta_{\varepsilon}(x)=
			\begin{cases}
				0, & \mbox{for} \ \varepsilon\leqslant|x| \leqslant 1/2\\
				1, & \mbox{for} \ |x|\leqslant\varepsilon/2 \ \mbox{or} \ |x|\geqslant3/4,
			\end{cases}
		\end{equation}
		and $|D^{(j)} \eta_{\varepsilon}(x)| \leqslant C \varepsilon^{-j}$ for $j=0,1,2,3,4,5,6$.
		Defining $\xi_{\varepsilon}=\eta_{\varepsilon}\phi$, multiplying \eqref{ourPDE} by $\xi_{\varepsilon}$ and integrating by parts in $B_{1}$, we obtain
		\begin{equation}\label{gs1}
			\int_{B_{1}} \eta_{\varepsilon} u \phi\left(\frac{(-\Delta)^{3} \phi}{\phi}-f_p(u)\right)\ud x=-\int_{B_{1}} u \mathfrak{T}\left(\eta_{\varepsilon}, \phi\right)\ud x,
		\end{equation}
		where $\mathfrak{T}_\varepsilon:C^{\infty}_c(B_1)\rightarrow C^{\infty}_c(B_1)$ is defined by
		\begin{equation*}
			\mathfrak{T}_\varepsilon(\phi)=\mathfrak{T}\left(\eta_{\varepsilon}, \phi\right)=6 \nabla \eta_{\varepsilon}\nabla \Delta^2\phi-15 \Delta \eta_{\varepsilon} \Delta^2\phi+20 \nabla \Delta\eta_{\varepsilon}\nabla \Delta\phi-15 \Delta^2\eta_{\varepsilon}\Delta\phi+6 \nabla \Delta^2\eta_{\varepsilon}\nabla\phi-\phi\Delta^{3} \eta_{\varepsilon}.
		\end{equation*}
		Using Lemma~\ref{lm:universalestimates} combined with the estimates on the cut-off function \eqref{cutoff2} and its derivatives, there exist $c_1,c_2>0$, independent of $\varepsilon$, satisfying the following estimates,
		\begin{equation*}
			\left|\int_{B_{1}} u\mathfrak{T}_\varepsilon(\phi)\ud x\right| \leqslant  c_{1}+c_{2} \varepsilon^{n} \varepsilon^{\zeta_p-6} \varepsilon^{-\gamma_p}<+\infty,
		\end{equation*}
		which implies that the right-hand side of \eqref{gs1} is uniformly bounded.
		In addition, assumption \eqref{condition} yields that $u^{p-1}(x)=\mathrm{o}(1)|x|^{-6}$ as $x\rightarrow 0$, which together with \eqref{gs4} and \eqref{gs2} provides that there exists $C>0$ satisfying
		\begin{equation}\label{gs2}
			\int_{B_{1}}\eta_{\varepsilon} u(x)|x|^{\zeta_p-6}\ud x\leqslant C.
		\end{equation}
		Therefore, by Lemma~\ref{lm:universalestimates}, it holds
		\begin{align}\label{gs3}
			\int_{\{\varepsilon \leqslant|x| \leqslant {1}/{2}\}} u(x)^{n\gamma_p^{-1}}\ud x  \leqslant C\int_{\{\varepsilon \leqslant|x| \leqslant {1}/{2}\}} u(x)|x|^{\zeta_p-6}\ud x\leqslant C\int_{B_{1}} \eta_{\varepsilon} u(x)|x|^{\zeta_p-6}\ud x<+\infty,
		\end{align}
		where the last inequality comes from \eqref{gs2}.
		Finally, passing to the limit as $\varepsilon \rightarrow 0$ in \eqref{gs3}, the proof of Claim 1 follows by applying the dominated convergence theorem.
		
		\noindent{\bf Claim 2:} If \eqref{condition} holds, then $u \in L^{q}(B_1)$ for all $q>2^{\#}$.
		
		\noindent Indeed, by Lemma~\ref{lm:integralrepresentation} (i), there exist a Green function with homogeneous Dirichlet boundary conditions $G_3(x,y)$ and $\psi\in L^{\infty}_{\rm loc}(B_{1/2})$ with $\Delta^2\psi=0$ and $\Delta\psi=0$ such that
		\begin{equation*}
			u(x)=\int_{B_1}G_2(x,y)\Delta^3u\ud x+\psi(x) \quad \mbox{in} \quad B_{1/2}.
		\end{equation*}
		More precisely, $G_{3}(x, y)$ is a distributional solution to the Dirichlet problem
		\begin{equation*}
			\begin{cases}
				\Delta^{3} G_{2}(x,y)=\delta_x(y) & \mbox{in} \quad B_{1/2}\\
				G_{3}(x, y)=\partial_{\nu}G_{3}(x, y)=\Delta G_{3}(x, y)=0 & \mbox{on} \quad \partial B_{1/2},
			\end{cases}
		\end{equation*}
		and there exists positive constant $C_{n}>0$ such that
		\begin{equation*}
			0<G_{3}(x, y) \leqslant \widetilde G_{3}(|x-y|):=C_{n}|x-y|^{6-n} \quad \mbox{for} \quad  x,y\in B_{1/2} \quad \mbox{and} \quad |x-y|>0,
		\end{equation*}
		where $\widetilde G_{3}(x,y)=C_n|x-y|^{6-n}$ is the fundamental solution to $\Delta^3$ in $\mathbb{R}^n$.
		Recall that $u\in C^6(B_1^*)$ satisfies
		\begin{equation}\label{potentialsystem}
			(-\Delta)^3 u=V(x) u \quad \mbox{in} \quad B_1^*,
		\end{equation} 
		where $V(x)=u^{p-1}$. 
		Moreover, using \eqref{condition}, we find that $V\in L^{n/6}(B_{1/2})$.
		
		Let us consider $Z=C^{\infty}_c(B_{1/4})$, $X=L^{2^{\#}}(B_{1/4})$ and $Y=L^{p}(B_{1/4})$ for $q\in(2^{\#},+\infty)$ as in \cite[Theorem~3.3.1]{MR1338474}. 
		Hence, it is well-defined the following inverse operator
		\begin{equation*}
			(Tu)(x)=\int_{B_{1/4}}\widetilde G_{3}(x,y)u(y)\ud y.
		\end{equation*}
		We also consider the operator $T_M:=\widetilde G_{3}\ast V_M$, which applied in both sides of \eqref{potentialsystem}, provides $u=T_Mu+\widetilde{T}_Mu$, where 
		\begin{equation*}
			(T_M u)(x)=\int_{B_{1/4}}\widetilde G_{3}(x,y)V_{M}(y)u(y)\ud y \quad \text{and} \quad (\widetilde{T}_M u)(x)=\int_{B_{1/4}}\widetilde G_{3}(x,y)\widetilde{V}_M(y)u(y)\ud y.
		\end{equation*}
		Here, for $M>0$, we define  $\widetilde{V}_M(x)=V(x)-V_M(x)$, where
		\begin{equation*}
			V_M(x)=
			\begin{cases}
				V(x), \ {\rm if} \ |V(x)|\geqslant M,\\
				0,\ {\rm otherwise}.
			\end{cases}
		\end{equation*}
		
		Now we can run the regularity lifting method, which is divided into two steps.
		
		\noindent{\bf Step 1:} For $q\in(2_{\#},+\infty)$, there exists $M\gg1$ large such that $T_M:L^{q}(B_{1/4})\rightarrow L^{q}(B_{1/4})$ is a contraction.
		
		\noindent In fact, for any $q\in(2_{\#},+\infty)$, there exists $m\in (1,n/6)$ such that $q=nm/(n-6m)$. Then, by the Hardy--Littlewood--Sobolev and H\"{o}lder inequalities \cite{MR717827}, for any $u\in L^{q}(\mathbb{R}^n)$, we get
		\begin{equation*}
			\|T_Mu\|_{L^{q}(B_{1/4})}\leqslant\|\widetilde G_{3}\ast V_{M}u\|_{L^{q}(B_{1/4})}\leqslant C\|V_M\|_{L^{{n}/{6}}(B_{1/4})}\|u\|_{L^{q}(B_{1/4})}.
		\end{equation*}
		Since $V_{M}\in L^{n/6}(B_{1/4})$ it is possible to choose a large $M\gg1$ satisfying $\|V_M\|_{L^{{n}/{6}}(B_{1/4})}<{1}/{2C}$. 
		Therefore, we arrive at
		\begin{equation*}
			\|T_Mu\|_{L^{q}(B_{1/4})}\leqslant{1}/{2}\|u\|_{L^{q}(B_{1/4})},
		\end{equation*}
		which implies that $T_M$ is a contraction.
		
		\noindent{\bf Step 2:} For any $q\in(2_{\#},+\infty)$, it follows that $\widetilde{T}_Mu\in L^{q}(B_{1/4})$.
		
		\noindent Indeed, for any $q\in(2_{\#},+\infty)$, we pick $1<m<n/6$ satisfying $q=nm/(n-6m)$. 
		Since $\widetilde{V}_M$ is bounded, we get
		\begin{equation*}
			\|\widetilde{T}_M\|_{L^{q}(B_{1/4})}=\|\widetilde G_{3}\ast\widetilde{V}_Mu\|_{L^{q}(B_{1/4})}\leqslant C\|\widetilde{V}_Mu\|_{L^{m}(B_{1/4})}\leqslant C\|u\|_{L^{m}(B_{1/4})}.
		\end{equation*}
		However, using \eqref{condition}, we have that  $u\in L^{q}(B_{1/4})$ for $q\in(1,n\gamma_p^{-1})$.
		Besides, $q=(p-2)n\gamma_p^{-1}$ when $m=n\gamma_p^{-1}$. Thus, we obtain that $u\in L^{q}(B_{1/4})$ for
		\begin{equation*}
			\begin{cases}
				1<q<\infty,& {\rm if} \ p\geqslant2,\\
				1<q\leqslant(2-p)^{-1}n\gamma_p^{-1},& {\rm if} \ 1<p<2.
			\end{cases}
		\end{equation*}
		Now we can repeat the argument for $m=(p-2)n\gamma_p^{-1}$ to get that $u\in L^{q}(B_{1/4})$ for
		\begin{equation*}
			\begin{cases}
				1<q<\infty,& {\rm if} \ p\geqslant2,\\
				1<q\leqslant(2-p)^{-1}n\gamma_p^{-1},& {\rm if} \ 1<p<2.
			\end{cases}
		\end{equation*}
		Therefore, by proceeding inductively as in \cite[Lemma~3.8]{MR4123335}, the proof of the claim follows.
		Ultimately, combining Steps 1 and 2, we can apply \cite[Theorem~3.3.1]{MR1338474} to show that $u\in L^{p}(B_{1/4})$ for all $p>2^{\#}$. In particular, the proof of the claim is finished.
		
		Now, by the Morrey's embedding theorem, it follows that $u\in C^{0,\alpha}(B_{1/4})$, for some $\alpha\in(0,1)$. Finally using Schauder estimates, one gets that $u\in C^{6,\alpha}(B_{1/4})$. 
		In particular, the singularity at the origin is removable, which concludes the proof of the lemma.
	\end{proof}
	
	\begin{proof}[Proof of Proposition~\ref{prop:gidasspruckcase}]
		Suppose by contradiction that $u\in C^6(\mathbb R^n\setminus\{0\})$ has a non-removable singularity at the origin, then by Lemma~\ref{lm:removablesingularitygidasspruck}, $u$ does not satisfy \eqref{asymptoticgidasspruck}. Therefore, the proof follows as a consequence of Lemma~\ref{lm:limitinglevesgidasspruck}.
	\end{proof}
	
	\subsection{Aviles case}\label{subsec:aviles}
	Finally, we prove Theorem~\ref{maintheorem} (b).
	The asymptotic analysis for the lower critical exponent, $p=2_{\#}$ exhibits its subtlety. 
	First, since $\gamma(2_{\#})=n-6$, one would expect the singular solutions to \eqref{ourPDE} to have the same behavior as the fundamental solution to the tri-Laplacian near the origin; thus, the isolated singularity would be removable.
	
	Our objective is to prove the proposition below
	\begin{proposition}\label{prop:avilescase}
		Let $u\in C^{6}(B_R^*)$ be a positive singular solution to \eqref{ourPDE} with $p=2_{\#}$. 
		Assume that $-\Delta u\geqslant0$ and $\Delta^2u\geqslant0$.
		Then, 
		\begin{equation*}
			u(x)=(1+\mathrm{o}(1))\widehat{K}_{0}(n)^{\frac{n-6}{6}}|x|^{6-n}(\ln|x|)^{\frac{6-n}{6}},
		\end{equation*}
		where $\widehat{K}_{0}(n):=-\lim_{t\rightarrow-\infty}t\widetilde{K}_0(n,t)$ is given by \eqref{serrinasymptoticsconstant}.
	\end{proposition}
	
	As in the autonomous case, we use the limiting energy levels $\widetilde{\mathcal{P}}_{\rm cyl}(-\infty ,w)$ to classify the local behavior near the isolated singularity. 
	\begin{lemma}\label{lm:limitinglevels}
		Let $u\in C^{6}(B_R^*)$ be a positive singular solution to \eqref{ourPDE} with $p=2_{\#}$ and $w=\widetilde{\mathfrak{F}}(v)$ be its nonautonomous Emden--Fowler transformation given by \eqref{newcylindrical}. 
		Assume that $-\Delta u\geqslant0$ and $\Delta^2u\geqslant0$.
		Then, $\widetilde{\mathcal{P}}_{\rm cyl}(-\infty,w) \in\{-\ell^*_{\#}, 0\}$, where $-\ell^*_{\#}=\frac{3}{2(n-3)}\widehat{K}_0(n)^{\frac{n-3}{3}}$.
		Moreover, it follows
		\begin{itemize}
			\item[{\rm (i)}] $\widetilde{\mathcal{P}}_{\rm cyl}(-\infty,w)=0$ if and only if,
			\begin{equation}\label{asymptoticsaviles}
				u(x)=\mathrm{o}\left(|x|^{6-n}(\ln |x|)^{\frac{6-n}{6}}\right) \quad {\rm as} \quad x \rightarrow 0.
			\end{equation}
			\item[{\rm (ii)}] $\widetilde{\mathcal{P}}_{\rm cyl}(-\infty,w)=\ell^*_{\#}$ if and only if, \begin{equation*}
				u(x)=(1+\mathrm{o}(1))\widehat{K}_{0}(n)^{\frac{n-6}{6}}|x|^{6-n}(\ln |x|)^{\frac{6-n}{6}} \quad {\rm as} \quad x \rightarrow 0.
			\end{equation*}
		\end{itemize}
	\end{lemma}
	
	\begin{proof}
		First, combining \eqref{angularestimates} with Proposition~\ref{lm:lowermonotonicity} and Lemma~\ref{lm:estimateangularparts}, we find
		\begin{equation*}
			\widetilde{\mathcal{P}}_{\rm cyl}(-\infty , w)=\lim _{t \rightarrow -\infty } \int_{\mathbb{S}_t^{n-1}}\left(\frac{n-6}{2(n-3)}|w|^{2_{\#}+1}+\widehat{K}_{0}(n)|w|^{2}\right)\ud\theta.
		\end{equation*}
		Furthermore, by \eqref{angularestimates}, we see that for any $\{t_k\}_{k\in\mathbb{N}}$ such that $t_{k} \rightarrow +\infty $ as $k \rightarrow -\infty $, it follows that $\{w(t_{k}, \theta)\}_{k\in\mathbb{N}}$ converges to a limit, which is independent of $\theta \in \mathbb{S}_t^{n-1}$. Hence, up to subsequence, there exists $w_0\in\mathbb{R}$ such that $w(t_{k}, \theta)\rightarrow w_0$ uniformly on $\theta \in \mathbb{S}_t^{n-1}$ , which gives us
		\begin{equation}\label{bestconstant}
			\widetilde{\mathcal{P}}_{\rm cyl}(-\infty , w)=\frac{n-6}{2(n-3)}|w_0|^{2_{\#}+1}+\widehat{K}_{0}(n)|w_0|^{2}.
		\end{equation}
		Thus, since the right-hand side of the last equation has at most three nonnegative roots, the limit $w_0\in\mathbb R$, under the uniform convergence of $w(t, \theta)$ on $\mathbb{S}_t^{n-1}$ as $t \rightarrow -\infty $, is unique.
		Finally, taking the inner product of \eqref{ourPDEcylnon} with $w$, integrating both sides over $(-\infty,t_{0},) \times \mathbb{S}^{n-1}$, and using  \eqref{angularestimates},\eqref{radialestimates} and Lemma~\ref{lm:estimateangularparts}, it follows
		\begin{equation*}
			\left|\int_{-\infty}^{t_{0}} \frac{1}{t} \int_{\mathbb{S}_t^{n-1}}\left(\widehat{K}_0(n)-|w|^{2_{\#}-1}\right) |w|^{2}\ud \theta\ud t\right|<+\infty.
		\end{equation*}
		Now since $\lim_{t \rightarrow -\infty }w(t,\theta)=w_0$ uniformly on $\mathbb{S}_t^{n-1}$, we get either $w_0=0$ or $w_0=\widehat{K}_0(n)^{\frac{n-6}{6}}$, which by substituting into \eqref{bestconstant}, implies that either $\widetilde{\mathcal{P}}_{\rm cyl}(-\infty , w)=0$ if and only if, $w_0=0$, or $\widetilde{\mathcal{P}}_{\rm cyl}(-\infty , w)=\ell^*_{\#}$, otherwise.
		The proof trivially follows by applying the inverse $\widetilde{\mathfrak{F}}^{-1}$ of the nonautonomous cylindrical transform.
	\end{proof}
	
	Now we are left to show that, if Lemma~\ref{lm:limitinglevels} (i)  holds, then the singularity at the origin is removable.
	Here, we are based on the barriers construction in \cite{MR2055032} (see also \cite{arxiv:1901.01678}), which is available due to the integral representation \eqref{integralsystem}.
	
	\begin{lemma}\label{lm:removablesingularityaviles}
		Let $u\in C^{6}(B_R^*)$ be a positive solution to \eqref{ourPDE} with $p=2_{\#}$. 
		Assume that $-\Delta u\geqslant0$ and $\Delta^2u\geqslant0$.
		If
		\begin{equation*}
			u(x)=\mathrm o\left(|x|^{6-n}(\ln |x|)^{\frac{6-n}{6}}\right) \quad {\rm as} \quad x \rightarrow 0,
		\end{equation*}
		then, the origin is a removable singularity.
	\end{lemma} 
	
	\begin{proof}
		For any $\delta>0$, we choose $0<\rho\ll 1$ such that $u(x)\leqslant\delta|x|^{-\gamma_p}$ in $B^*_{\rho}$.
		Fixing $\varepsilon>0$, $\kappa\in\left(0, \gamma_p\right)$ and $M\gg1$ to be chosen later, we define
		\begin{equation*}
			\varsigma(x)=
			\begin{cases}
				{M|x|^{-\kappa}+\varepsilon|x|^{6-n-\kappa},} & \mbox{if} \ {0<|x|<\rho},\\ 
				{u(x)},& \mbox{if} \ {\rho<|x|<2}.
			\end{cases}
		\end{equation*}
		Notice that for every $0<\kappa<n-6$ and $0<|x|<2$, there exists $C>0$ such that
		\begin{align*}
			\int_{\mathbb{R}^{n}}{|x-y|^{6-n}|y|^{-6-\kappa}}\ud y 
			=|x|^{6-n}\int_{\mathbb{R}^{n}}{\left||x|^{-1}x-|x|^{-1}y\right|^{6-n}|y|^{\-\kappa-6}}\ud y\leqslant C\left(\frac{1}{n-6-\kappa}+\frac{1}{\kappa}+1\right)|x|^{-\kappa},
		\end{align*}
		which, for $0<|x|<2$ and $0<\delta\ll1$, yields 
		\begin{align*}
			\int_{B_{\rho}}{u^{2_{\#}-1}(y)\varsigma(y)}{|x-y|^{6-n}}\ud y
			\leqslant \delta^{2_{\#}-1} \int_{\mathbb{R}^{n}}{\varsigma(y)}{|x-y|^{n-6}|y|^{-6}}\ud y\leqslant C\delta^{2_{\#}-1} \varsigma(x)<\frac{1}{2} \varsigma(x).
		\end{align*}
		Moreover, for $0<|x|<\rho$ and $\bar{x}=\rho x|x|^{-1}$, we get
		\begin{align*}
			\int_{B_{2}\setminus B_{\rho}}{u^{2_{\#}-1}(y)\varsigma(y)}{|x-y|^{6-n}}\ud y
			=\int_{B_{2}\setminus B_{\rho}}\frac{|\bar{x}-y|^{n-6}}{|x-y|^{n-6}}\frac{u^{2_{\#}}(y)}{|\bar{x}-y|^{n-6}}\ud y\leqslant  2^{n-6}\max_{\partial B_{\rho}} u.
		\end{align*}
		The last inequality implies that for $0<|x|<\tau$ and $M\geqslant\max_{\partial B_{\rho}}u$,
		\begin{equation*}
			\psi(x)+\int_{B_{2}}\frac{u^{2_{\#}-1}(y)\varsigma(y)}{|x-y|^{6-n}}\ud y \leqslant \psi(x)+2^{n-6}\max_{\partial B_{\rho}}u+\frac{1}{2}\varsigma(x)<\varsigma(x).
		\end{equation*}
		
		We show that $\varsigma$ can be taken as a barrier for any $u$. 
		Namely, we claim that $u(x)\leqslant \varsigma(x)$ in $B^*_{\rho}$. In fact, suppose by contradiction that the conclusion is not true. Then, since $u(x)\leqslant\delta|x|^{-\gamma_p}$ in $B^*_{\rho}$, by the definition of $\varsigma$, there exists $\widetilde{\tau} \in(0, \rho)$, depending on $\varepsilon$, such that $\varsigma\geqslant u$ in $B^*_{{\rho}}$ and $\varsigma>u$ close to the boundary $\partial B_{\rho}$. 
		Let us consider $\bar{\tau}:=\inf\{\tau>1 : \tau\psi>u \ \mbox{in} \ B^*_{\rho}\}$.
		Then, we have that $\bar{\tau}\in(1,+\infty)$ and there exists $\bar{x}\in B_{\rho} \setminus\bar{B}_{\widetilde{\tau}}$ such that $\bar{\tau}\varsigma(\bar{x})=u(\bar{x})$ and, for    $0<|x|<\tau$, it follows 
		\begin{equation*}
			\bar{\tau}\varsigma(x)\geqslant\int_{B_{2}}{u^{2_{\#}-1}(y)\bar{\tau} \varsigma(y)}{|x-y|^{6-n}}\ud y+\bar{\tau}\psi(x)\geqslant\int_{B_{2}}{u^{2_{\#}-1}(y)\bar{\tau}\varsigma(y)}{|x-y|^{6-n}}\ud y+\psi(x),
		\end{equation*}
		which gives us
		\begin{equation*}
			\bar{\tau}\varsigma(x)-u(x)\geqslant\int_{B_{2}}{u^{2_{\#}-1}(y)(\bar{\tau} \varsigma(y)-u(y))}{|x-y|^{6-n}}\ud y.
		\end{equation*}
		Finally, by evaluating the last inequality at $\bar{x}\in B_{\rho}\setminus\bar{B}_{\widetilde{\tau}}$, we get a contradiction.
		
		At last, we find
		$u(x)\leqslant\varsigma(x)\leqslant M|x|^{-\kappa}+\varepsilon|x|^{6-n-\kappa}$ in $B^*_{\rho}$,
		which yields that $u^{2_{\#}-1}\in L^{p}(B^*_{\rho})$ for some $p>{n}/{6}$. 
		Hence, standard elliptic regularity concludes the proof of the lemma.
	\end{proof}
	
	Ultimately, the proof of the main result in this section is merely a consequence of the last results.
	
	\begin{proof}[Proof of Proposition~\ref{prop:avilescase}]
		Suppose that $u\in C^6(\mathbb R^n\setminus\{0\})$ is a positive singular solution to \eqref{ourPDE} with $p=2_{\#}$, then by Lemma~\ref{lm:removablesingularityaviles}, $u$ does not satisfy \eqref{asymptoticsaviles}. Therefore, the proof follows as a consequence of Lemma~\ref{lm:limitinglevels}.
	\end{proof}
	
	\begin{acknowledgement}
		This paper was finished when the first-named author held a Post-doctoral position at the University of British Columbia, whose hospitality he would like to acknowledge.
		He also wishes to express gratitude to Professor Jo\~ao Marcos do \'O for his constant support and several valuable conversations.
	\end{acknowledgement}
	
	
	\appendix
	
	\section{Tri-Laplacian coefficients}
	In this appendix, we show the explicit formula for the coefficients of the tri-Laplacian in autonomous and nonautonomous Emden--Fowler coordinates, respectively.
	
	\subsection{Polar coordinates}
	We consider the cylinder $\mathcal{C}_R:=(0,R)\times\mathbb{S}^{n-1}$ and $(-\Delta)^3_{\rm sph}$ the tri-Laplacian written in spherical (polar) coordinates given by
	\begin{align*}
		\Delta^3_{\rm sph}&={r^{-6}}\partial_r^{(6)}+M_5(n,r)\partial_r^{(5)}+M_4(n,r)\partial_r^{(4)}+ M_3(n,r)\partial_r^{(3)}+M_2(n,r)\partial_r^{(2)}+M_1(n,r)\partial_r&\\\nonumber
		&+2r^{-2}\partial^{(4)}_r\Delta_{\sigma}+N_3(n,r)\partial^{(3)}_r\Delta_{\sigma}+N_2(n,r)\partial^{(2)}_r\Delta_{\sigma}+N_1(n,r)\partial_r\Delta_{\sigma}+N_0(n,r)\Delta_{\sigma}\\
		&+3r^{-4}\partial^{(2)}_r\Delta^2_{\sigma}+O_1(n,r)\partial_r\Delta^2_{\sigma}+O_0(n,r)\Delta^2_{\sigma}+{r^{-6}}\Delta^3_{\sigma},
		&
	\end{align*}
	where $\Delta_{\sigma}$ denotes the Laplace--Beltrami operator in $\mathbb{S}^{n-1}$ and
	\begin{align}\label{laplaciancoefficients1}
		\nonumber 
		M_{5}(n,r)&=3(n-1)r^{-1}\\\nonumber
		M_{4}(n,r)&=3(n-1)(n-3)r^{-2}\\
		M_{3}(n,r)&=(n-1)(n-3)(n-8)r^{-3}\\\nonumber
		M_{2}(n,r)&=-3(n-1)(n-3)(n-5)r^{-4}\\\nonumber
		M_{1}(n,r)&=3(n-1)(n-3)(n-5)r^{-5}\\\nonumber
	\end{align}
	and
	\begin{align}\label{laplaciancoefficients2}
		\nonumber 
		N_{3}(n,r)&=2(n-7)r^{-3}\\\nonumber
		N_{2}(n,r)&=2(n^2-n-3)r^{-4}\\\nonumber
		N_{1}(n,r)&=6(7n-23)r^{-5}\\
		N_{0}(n,r)&=8(n-1)(n-5)r^{-6}\\\nonumber
		O_{1}(n,r)&=3(n-5)r^{-5}\\\nonumber
		O_{0}(n,r)&=-2(3n-16)r^{-6}\\\nonumber
	\end{align}
	are its coefficients in this coordinate system.
	
	\subsection{Autonomous case}
	Let us recall the sixth order autonomous Emden--Fowler change of variables $($or cylindrical logarithm coordinates$)$ given by
	\begin{equation*}
		v(t,\theta)=r^{\gamma_p}u(r,\sigma), \quad \mbox{where} \quad t=\ln r \quad {\rm and} \quad \sigma=\theta=x|x|^{-1}.
	\end{equation*}  
	Using this change of variables and performing a lengthy computation, we arrive at the following sixth order operator on the cylinder,
	\begin{align*}
		P^3_{\rm cyl}&=\partial_t^{(6)}+K_{5}(n,p)\partial_t^{(5)}+K_{4}(n,p)\partial_t^{(4)}+K_{3}(n,p)\partial_t^{(3)}+K_{2}(n,p)\partial_t^{(2)}+K_{1}(n,p)\partial_t+K_{0}(n,p)&\\\nonumber
		&+2\partial^{(4)}_t\Delta_{\theta}+J_3(n,p)\partial^{(3)}_t\Delta_{\theta}+J_2(n,p)\partial^{(2)}_t\Delta_{\theta}+J_1(n,p)\partial_t\Delta_{\theta}+J_0(n,p)\Delta_{\theta}\\
		&+3\partial^{(2)}_t\Delta^2_{\theta}+L_1(n,p)\partial_t\Delta^2_{\theta}+L_0(n,p)\Delta^2_{\theta}+\Delta^3_{\theta},&
	\end{align*}
	where
	\begin{align}\label{autonomouscoefficients1}
		\nonumber K_{0}(n,p)&=-24(p-1)^{-6}(p+2) (2 p+1) (n p- 6 p-n ) (n p- 4 p -2 - n) (n p- 2 p - n-4)&\\\nonumber 
		K_{1}(n,p)&=4(p-1)^{-5}(n p- 6 p- n-6) (2 n^2 p^4- 12 n p^4+ 16 p^4+ 10 n^2 p^3- 108 n p^3+ 224 p^3\\
		&-15 n^2 p^2- 36 n p^2+ 492 p^2- 8 n^2 p+ 120 n p+ 224 p+ 11 n^2+ 36 n +16)&\\\nonumber
		K_{2}(n,p)&=-2(p-1)^{-4} (3 n^3p^4- 48 n^2 p^4+ 228 n p^4- 320 p^4- 3 n^3 p^3- 78 n^2 p^3+ 996 n p^3- 2392 p^3- 9 n^3 p^2\\\nonumber
		&+ 198 n^2 p^2+ 180 n p^2- 4296 p^2+ 15 n^3 p+ 30 n^2 p- 1020 n p- 2392 p- 6 n^3- 102 n^2- 384 n -320)&\\\nonumber
		K_{3}(n,p)&=-(p-1)^{-3}(n p- 6 p-6 - n) (n^2 p^2- 24 n p^2+ 68 p^2- 2 n^2 p- 12 n p+ 224 p+ n^2+ 36 n+ 68)&\\\nonumber
		K_{4}(n,p)&=(p-1)^{-2}(3 n^2 p^2- 42 n p^2+ 124 p^2- 6 n^2 p- 6 n p+ 292 p+ 3 n^2+ 48 n+124)&\\\nonumber
		K_{5}(n,p)&=3(p-1)^{-1} (n p- 6 p-6 - n)&
	\end{align}
	and
	\begin{align}\label{autonomouscoefficients2}
		\nonumber
		J_{0}(n,p)&=4(p-1)^{-4} (2 n^2 p^4- 12 n p^4+ 10 p^4- 5 n^2 p^3- 24 n p^3+ 218 p^3+ 21 n^2 p^2+ 72 n p^2- 192 p^2&\\\nonumber
		&- 35 n^2 p- 132 n p+ 1094 p+ 17 n^2+ 96 n-482)&\\\nonumber
		J_{1}(n,p)&=-2(p-1)^{-3}(n^2 p^3- 24 n p^3+ 86 p^3+ 9 n^2 p^2+ 24 n p^2+ 90 p^2- 21 n^2 p- 84 n p+ 966 p &\\\nonumber
		&+ 11 n^2+ 84 n-278)&\\
		J_{2}(n,p)&=2(p-1)^{-2}(n^2 p^2- 4 n p^2+ 29 p^2- 2 n^2 p- 10 n p+ 176 p+ n^2+ 14 n+11)&\\\nonumber
		J_{3}(n,p)&=2(p-1)^{-1} (n p - 13 p - n-11)&\\\nonumber
		L_{0}(n,p)&=-(p-1)^{-2}(3 n p^2 16 p^2-3 n p+22 p +6 n+16    )&\\\nonumber
		L_{1}(n,p)&=3(p-1)^{-1} (n p- 6 p -6 - n ).&
	\end{align}
	
	\subsection{Nonautonomous case}
	Let us recall the sixth order nonautonomous Emden--Fowler change of variables $($or cylindrical logarithm coordinates$)$ given by
	\begin{equation*}
		w(t,\theta)=r^{6-n}(\ln r)^{\frac{6-n}{6}}u(r,\sigma), \quad {\rm where} \quad t=\ln r \quad {\rm and} \quad \sigma=\theta=x|x|^{-1}.
	\end{equation*}
	Using this coordinate system and performing a lengthy computation, we arrive at the following sixth order nonautonomous operator PDE on the cylinder
	\begin{align*}
		\widetilde{P}^3_{\rm cyl}&=\partial_t^{(6)}+\widetilde{K}_{5}(n,t)\partial_t^{(5)}+\widetilde{K}_{4}(n,t)\partial_t^{(4)}+\widetilde{K}_{3}(n,t)\partial_t^{(3)}+\widetilde{K}_{2}(n,t)\partial_t^{(2)}+\widetilde{K}_{1}(n,p)\partial_t+\widetilde{K}_{0}(n,t)&\\\nonumber
		&+2\partial^{(4)}_t\Delta_{\theta}+\widetilde{J}_3(n,t)\partial^{(3)}_t\Delta_{\theta}+\widetilde{J}_2(n,t)\partial^{(2)}_t\Delta_{\theta}+\widetilde{J}_1(n,t)\partial_t\Delta_{\theta}+\widetilde{J}_0(n,t)\Delta_{\theta}\\
		&+3\partial^{(2)}_t\Delta^2_{\theta}+\widetilde{L}_1(n,t)\partial_t\Delta^2_{\theta}+\widetilde{L}_0(n,t)\Delta^2_{\theta}+\Delta^3_{\theta},&
	\end{align*}
	where
	\begin{align}\label{coeficcientnonautonomous1}
		\nonumber
		\widetilde{K}_{0}(n,t)&=\frac{4(n-6) (n^3-12n^2+44n-48)}{3t}-\frac{(n-6) n (3n^3-48n^2+228n-320)}{18t^{2}}&\\\nonumber
		&+\frac{(n-6) n (n^4-24n^3+32n^2+864n-2448)}{216t^{3}}+ \frac{(n-6) n (3n^4+12n^3-416n^2-792n+8928)}{1296t^{4}} &\\\nonumber 
		&+ \frac{(n-6) n (n^4+30n^3+180n^2-1080n-7776)}{2592t^{5}}+\frac{(n-6) n (n^4+60n^3+1260n^2+10800n+31104)}{46656t^{6}}&\\\nonumber
		\widetilde{K}_{1}(n,t)&=-8(n^3- 12 n^2+ 44 n-48)
		+\frac{2(n^4-66n^3+516n^2-1688n^3+1920)}{3t}&\\
		&- \frac{(n-6)^2 n (n^2-24n+68)}{12t^2}-\frac{n(3n^4-42 n^3+16n^2+1512n-4464)}{54t^3} &\\\nonumber
		&- \frac{5 (n-6)^2 n (n^2+ 18 n + 72 )}{432t^4}-\frac{n (n^4+30 n^3 +  180 n^2- 1080 n-7776 )}{1296 t^5}&\\\nonumber
		\widetilde{K}_{2}(n,t)&=-(6n^3-96n^2+456n-640)+ \frac{(n-6)^2 (n^2-24n+68)}{2t}&\\\nonumber
		&+\frac{n (3n^3-60n^2+376n-744)}{6t^2}+\frac{5(n-6)^2 n (n+6)}{36t^3}+ \frac{5n(n^3+12n^2-36n-432)}{432t^4}&\\\nonumber
		\widetilde{K}_{3}(n,t)&= -(n^3-30n^2+212n-408)-\frac{(6n^3-120n^2+752n-3496)}{3t}  - \frac{5 (n-6)^2 n}{6t^2}-\frac{5n(n^2-36)}{54t^3}&\\\nonumber
		\widetilde{K}_{4}(n,t)&=(3n^2-42n+124)+\frac{5(n-6)^2}{2t}+\frac{5n(n-6)}{12t^2} &\\\nonumber
		\widetilde{K}_{5}(n,t)&=- 3 (n-6)+\frac{n-6}{t}&
	\end{align}
	and
	\begin{align}\label{coeficcientnonautonomous2}
		\nonumber
		\widetilde{J}_{0}(n,t)&=2 (n^4- 8 n^3- 39 n^2+ 470 n-964)+	\frac{(3 n^4 - 34 n^3 +34 n^2+ 650 n -1668)}{3t}&\\\nonumber
		&+ \frac{n (4 n^3 - 43 n^2+ 125 n -66)}{18t^2}+ \frac{n (n^3- 11 n^2- 108 n +396)}{108t^3} &\\\nonumber
		&+\frac{n (n^3+ 12 n^2- 36 n -432)}{648 t^4}&\\\nonumber
		\widetilde{J}_{1}(n,t)&=-(6 n^3-32 n^2-124 n+556)+
		\frac{2(4 n^3- 43 n^2 + 125 n -66 )}{3t}&\\\nonumber
		&- \frac{n(3 n^2- 29 n + 66 )}{6 t^2}-\frac{n( n^2-36)}{27t^3}&\\
		\widetilde{J}_{2}(n,t)&=(8 n^2- 38 n + 22 )+\frac{(3 n^2- 29 n + 66 )}{t} + \frac{n(n-6)}{3t^2} &\\\nonumber
		\widetilde{J}_{3}(n,t)&=-\frac{4(n-6)}{3}-(6n-22)t&\\\nonumber
		\widetilde{L}_{0}(n,t)&=-\frac{(72n-384)}{12}+\frac{(n-6)^2}{2t} + \frac{n(n-6)}{12 t^2}&\\\nonumber
		\widetilde{L}_{1}(n,t)&=- 3 (n-6)-\frac{n-6}{t}. &
	\end{align}
	
	Here, the real-valued functions $\mathfrak{p}_j(n,\cdot):(-\infty,\ln R)\rightarrow\mathbb R$ for $j=0,1,2,3,4,5$ are given by
	\begin{align}\label{coefficientspohozaev}
		\nonumber
		\mathfrak{p}_5(n,t)&:=-\frac{3}{2}(n-6),&\\	\nonumber
		\mathfrak{p}_4(n,t)&:=-\frac{5}{2}{n(n-6)}t^{-2}+\frac{1}{2}(3 n^2-42n+124) ,&\\\nonumber
		\mathfrak{p}_3(n,t)&:=\frac{5}{12}n(n^2-36)t^{-3}+\frac{5}{12}n(n-6)t^{-2}-\frac{1}{2}(n^3-30n^2+212n-408),&\\	\nonumber
		\mathfrak{p}_2(n,t)&:=-\frac{5}{288}n(n-6)t^{-4}
		-\frac{5}{36}(n-6)^2n(n+6)t^{-3}-\frac{1}{12}n(3n^3-60n^2+376n-744)t^{-2}\\\nonumber
		&-(3n^3-48n^2+228n-320),\\
		\mathfrak{p}_1(n,t)&:=\frac{(n-6)}{864}\left[(64512 - 44928 n + 10368 n^2 - 864 n^3)t^{-5}\right.&
		\\\nonumber
		&+(553392 n + 170532 n^2 + 15732 n^3 + 447 n^4)t^{-4}+(9744 n - 2228 n^2 - 324 n^3 + 53 n^4)t^{-3}&\\\nonumber
		&+(-8736 n + 6608 n^2 - 1272 n^3 + 60 n^4)t^{-2}+(-14688 n + 7632 n^2 - 1080 n^3 + 36 n^4)t^{-1}&\\\nonumber
		&\left.-(864n^3-10368n^2+44928n-64512)+(3456n^2-20736n+27648)t\right],&\\\nonumber
		\mathfrak{p}_0(n,t)&:=-\frac{5}{93312}n(n-6)(n^4+60n^3+1260n^2+10800n+31104)t^{-6}&\\\nonumber
		&+\frac{1}{1296}n(n-6)(n^4+30n^3+180n^2-1080n-7776)t^{-5}&\\\nonumber
		&+\frac{5}{864}n(n-6)(3n^4+12n^3-416n^2-792n+8928)t^{-4}&\\\nonumber
		&+\frac{5}{216}n(n-6)(n^4-24n^3+32n^2+864n-2448)t^{-3}+\frac{1}{36}(3n^3-48n^2+228n-320)t^{-2}.&
	\end{align}
	
	\subsection{Upper critical case}
	When $p=2^{\#}-1$, it follows 
	\begin{align}\label{coeficcientsupper}
		&\nonumber K^{\#}_{0}(n)=-2^{-8}(n-6)^2(n-2)^2(n+2)^2&\\\nonumber 
		&K^{\#}_{1}(n)=0&\\\nonumber
		&K^{\#}_{2}(n)=2^{-4}(3n^4-24n^3+72n^2-96n+304)&\\\nonumber
		&K^{\#}_{3}(n)=0&\\\nonumber
		&K^{\#}_{4}(n)=-2^{-2}(3n^2-12n+44)&\\\nonumber
		&K^{\#}_{5}(n)=0&\\
		&J^{\#}_{0}(n)=2^{-3}(3n^4-18n^3-192n^2+1864n-3952)&\\\nonumber
		&J^{\#}_{1}(n)=-2^{-1}(3n^3+3n^2-244n+620)&\\\nonumber
		&J^{\#}_{2}(n)=2 n^2+13n-68&\\\nonumber
		&J^{\#}_{3}(n)=-2 (n+1)&\\\nonumber
		&L^{\#}_{0}(n)=-2^{-2}(3 n^2-12n-20)&\\\nonumber
		&L^{\#}_{1}(n)=0.&\\\nonumber
	\end{align}
	
	\subsection{Lower critical case}
	When $p=2_{\#}$, it follows
	\begin{align}\label{coeficcientslower}
		&\nonumber K_{0,\#}(n)=0&\\\nonumber 
		&K_{1,\#}(n)=-8 (n-6) (n-4) (n-2)&\\\nonumber
		&K_{2,\#}(n)=-2 (3n^3-48n^2+228n-320)&\\\nonumber
		&K_{3,\#}(n)=-(n-6) (n^2-24n+68)&\\\nonumber
		&K_{4,\#}(n)=3 n^2-42n+124&\\\nonumber
		&K_{5,\#}(n)=-3 (n-6)&\\
		&J_{0,\#}(n)=2 (n^4-8n^3-39n^2+470n-964)&\\\nonumber
		&J_{1,\#}(n)=-2 (3 n^3-16n^2-62n+278)&\\\nonumber
		&J_{2,\#}(n)=2 (4 n^2-19n+11)&\\\nonumber
		&J_{3,\#}(n)=-2 (3 n-11)&\\\nonumber
		&L_{0,\#}(n)=-2 (3 n-16)&\\\nonumber
		&L_{1,\#}(n)=-3 (n-6).
		&\\\nonumber
	\end{align}
	

\end{document}